\documentclass[12pt]{amsart}
\usepackage{amsfonts,amssymb,amscd,xy,mathtools,tikz-cd}
\usepackage[leqno]{amsmath}
\usepackage{mathrsfs}
\usepackage{amssymb}
\usepackage{amsmath}
\usepackage{amsxtra}
\usepackage{amsthm}
\usepackage{hyperref}
\newcommand{\picor}{{\rm Pic}_{\lbe X}^{\e 0\le,\e\red}}

\def\ms{\mathscr }

\def\msdiv{\mathscr Div}

\newcommand{\ksep}{k^{\e\rm s}}
\newcommand{\lcm}{{\rm lcm}}

\newcommand*{\defeq}{\mathrel{\vcenter{\baselineskip0.5ex \lineskiplimit0pt
			\hbox{\scriptsize.}\hbox{\scriptsize.}}}%
	=}

\def\lbe{\kern -.025em}
\def\picxs{{\rm Pic}_{\, X\!/\lbe S}}
\newcommand\ses{S_{\et}^{\e\sim}}

\def\br{{\rm{Br}}\e}

\def\bra{{\rm{Br}}_{\rm{a}}\le X}

\def\brap{{\rm{Br}}_{\rm{a}}\le\widetilde{X}}

\def\bro{{\rm{Br}}_{1}X}

\def\brop{{\rm{Br}}_{1}\widetilde{X}}

\def\brok{{\rm{Br}}_{\lle 1}X}
\def\brokp{{\rm{Br}}_{\lle 1}\widetilde{X}}

\def\picxs{{\rm Pic}_{\lbe X\!/\lbe S}}
\def\picxst{{\rm Pic}_{ \widetilde{X}\lbe\lbe/\lbe S}}
\def\picxkt{{\rm Pic}_{\widetilde{X}\be/k}}
\def\picxkn{{\rm Pic}_{X^{\be N}\!\be/k}}

\def\brxs{{\rm{Br}}_{\lbe X\be/\lbe S}}

\def\bxs{\br\lbe(\lbe X\!/\be S\e)}
\def\bys{\br(\le Y\be\lbe/\lbe S\le\lle)}
\def\bysp{\br(\le \widetilde{Y}\!/\be S\le\lle)}
\def\bxk{\br\lbe(\lbe X\be/ k\lle)}

\def\bxsp{\br\lbe(\lbe\widetilde{X}\be/\lbe S\le)}

\def\bxkp{\br\lbe\lle(\lbe\widetilde{X}\be/ k\le)}

\def\picxk{{\rm Pic}_{X\be/ k}}

\usepackage{anyfontsize}
\fontsize{1}{2}

\DeclareMathAlphabet{\mathbbmsl}{U}{bbm}{m}{sl}

\newcommand{\car}{{\rm char}\,}

\def\g{\varGamma}

\newcommand{\into}{\hookrightarrow}
\newcommand{\onto}{\twoheadrightarrow}
\newcommand{\isoto}{\overset{\!\sim}{\to}}

\def\kbar{\overline{k}}

\newcommand\red{{\rm red}}

\newcommand{\lra}{\longrightarrow}

\newcommand{\et}{{\rm {\acute et}}}

\def\be{\kern -.1em}
\def\le{\kern 0.03em}
\def\lle{\kern 0.017em}
\def\mle{\kern 0.009em}

\def\lbe{\kern -.05em}

\newcommand{\A}{{\mathbb A}}
\newcommand{\Z}{{\mathbb Z}}
\newcommand{\Q}{{\mathbb Q}}
\newcommand{\N}{{\mathbb N}}
\newcommand{\G}{{\mathbb G}}
\newcommand{\Pn}{{\mathbb P}}

\newcommand{\spec}{\mathrm{ Spec}\,}

\newcommand{\krn}{\mathrm{Ker}\,}
\newcommand{\img}{\mathrm{Im}\e\le}

\newcommand{\cok}{\mathrm{Coker}\,}

\def\e{\kern 0.08em}

\newcommand{\xs}{X^{\lle\rm{s}}}

\usepackage{tabularx}
\usepackage{epsfig,amssymb}
\usepackage{bm} 
\usepackage{verbatim} 
\usepackage{mathrsfs}

\usepackage{hyperref}

\definecolor{labelkey}{rgb}{1,0,0}

\makeatletter

\DeclareMathAlphabet{\mathcalligra}{T1}{calligra}{m}{n}

\numberwithin{equation}{section}

\def\pic{{\rm{Pic}}\e}

\usepackage{hyperref}

\definecolor{labelkey}{rgb}{1,0,0}

\numberwithin{equation}{section}

\newcommand\cO{{\mathcal O}}

\newtheorem{lemma}{Lemma}[section]
\newtheorem{theorem}[lemma]{Theorem}
\newtheorem{proposition-definition}[lemma]{Proposition-Definition}
\newtheorem{corollary}[lemma]{Corollary}
\newtheorem{proposition}[lemma]{Proposition}
\theoremstyle{definition}

\newtheorem{lemma-definition}[lemma]{Lemma-Definition}
\theoremstyle{remark}
\newtheorem{remark}[lemma]{Remark}
\newtheorem{remarks}[lemma]{Remarks}
\newtheorem{example}[lemma]{Example}

\begin{document}

\input xy     
\xyoption{all} 

\dedicatory{To Timothy J.\le Ford\le, for his book \cite{sa}.}

\title[The algebraic Brauer group of a pinched variety]{The algebraic Brauer group of a pinched variety}

\author{Cristian D. Gonz\'alez-Avil\'es}
\address{Departamento de Matem\'aticas, Universidad de La Serena, Cisternas 1200, La Serena 1700000, Chile}
\email{cgonzalez@userena.cl}
\thanks{The author is partially supported by Fondecyt grant 1200118.}
\date{\today}

\makeatletter
\@namedef{subjclassname@2020}{\textup{2020} Mathematics Subject Classification}
\makeatother


\maketitle

\topmargin -1cm

\smallskip

\begin{abstract} Let $k$ be any field, let $\widetilde{X}$ be a projective and geometrically integral $k$-scheme and let $\widetilde{Y}$ be a
finite closed subscheme of $\widetilde{X}$. If $\psi\colon \widetilde{Y} \to Y\!$ is a schematically dominant morphism between finite $k$-schemes and $X$ is obtained by pinching $\widetilde{X}$ along $\widetilde{Y}$ via $\psi$, we describe the kernel (and, in certain cases, the cokernel\le) of the induced pullback map $\br_{\be 1} X\to \br_{\be 1} \widetilde{X}$ between the corresponding algebraic (cohomological) Brauer groups of $X$ and $\widetilde{X}$ {\it solely} in terms of the Brauer groups of the residue fields of $Y$ and $\widetilde{Y}$ and the Amitsur subgroups of $X$ and $\widetilde{X}$ in $\br\le k$. As an application, we compute the algebraic Brauer group of a projective and geometrically integral $k$-scheme $X$ with a finite non\le-\le normal locus whose normalization $X^{\be N}$ is $k\le$-\le isomorphic to $\Pn_{\! k}^{\e{\rm dim}\le X}\!$. If $k$ is a local field and $X$ is a projective and geometrically integral $k$-curve such that $X^{\be N}$ is {\it smooth}, then we show that the order of the Amitsur subgroup of $X$ in $\br\e k$ is the index of $X$, i.e., the least positive degree of a $0\e$-\lle cycle on $X$. This  statement generalizes a well-known theorem of Roquette and Lichtenbaum, who obtained the above conclusion when $k$ is a $p\e$-\le adic field and $X$ is smooth over $k$.
\end{abstract}

\section*{0. Introduction}

The search for rational (or integral) solutions of concrete polynomial equations often involves singular varieties. For example, in their treatment of the Erd\H{o}s-Straus conjecture via the Brauer-Manin obstruction \cite{bl20}, Bright and Loughran considered a certain singular $\Q\e$-\e surface $X$ that is naturally associated to the conjecture and, in order to compute the indicated obstruction, determined its Brauer group $\br X$. In general, computing the (cohomological) Brauer group of a singular variety is a difficult problem. Consequently, the literature on this subject is rather sparse, especially over an imperfect field. See \cite[\S\S8.1-8.4]{ctsk} for a discussion of various examples over a field which is either algebraically closed or of characteristic $0$ and [op.cit., \S8.5] for a study of the Brauer group of a reduced and separated curve over a field of characteristic $0$. A natural method for studying the Brauer group of a singular variety $X$ in terms of a desingularization $\widetilde{X}\to X$ (when one is available) is to first compute the kernel of the pullback map $\br X\to\br \widetilde{X}$ induced by $\widetilde{X}\to X$ and then apply Grothendieck's purity theorem \cite{ces19} to the smooth variety $\widetilde{X}$ in order to obtain information about $\br \widetilde{X}$. This is, essentially, the approach adopted in \cite{bl20}. If no desingularization of $X$ is available, the above method will retain its value if a covering variety $\widetilde{X}$ of $X$ can be found such that the structure of $\br\widetilde{X}$ is more accessible than that of $\br X$. In this paper we implement this idea to study the {\it algebraic} Brauer group of a projective and geometrically integral $k\e$-\le scheme  with a {\it finite} non\le-\le normal locus over an {\it arbitrary} field $k$. To be precise, let $f\colon X\to \spec k$ be a proper and geometrically integral $k$-scheme, let $\ksep$ be a fixed separable closure of $k$ and set $X^{\rm s}=X\times_{k}\ksep$. The structure of the algebraic Brauer group $\br_{\be 1}X=\krn\!\be\left[\e\br X\be\to\be\br\xs\e\right]$ of $X$ is closely related to that of the Picard scheme $\picxk$ via a canonical exact sequence of abelian groups
\begin{equation}\label{brpic}
0\to \pic\le X\to (\pic\le\xs\le)^{\lbe\g}\to\br\e k\overset{f^{\lle *}}{\to} \br_{\lbe 1}X\to H^{\le 1}_{\et}\lbe(k,\picxk\lbe)\overset{\!\!d}{\to} H^{\le\lle 3}_{\et}\lbe(k,\G_{m}\lbe),
\end{equation}
where $\g={\rm Gal}\le(\ksep\be/k\lle)$. Unfortunately, the structure of $\picxk$ is unknown in general and determining this structure, especially over an imperfect field $k$, is a difficult unsolved problem. However, two important advances on this problem have been made in the past few years under certain restrictions. Firstly, Geisser \cite{gei} advanced our understanding of the structure of the Picard variety of an arbitrary reduced and proper scheme $X$ over a {\it perfect} field $k$. His results imply, {\it inter alia}, that if $\nu=\nu_{X}\colon  X^{\lbe N}\to X$ is the normalization morphism of $X$, then the structure of $\krn\be[\nu^{\le *}\be\colon\be \br_{\be 1} X\to \br_{\be 1} X^{\lbe N}\e]$ is closely related to the $\g$-module structure of $H^{1}_{\et}\lbe(\be X^{\rm{s}},\Z\le)\,$\footnote{\,We will not discuss this structure here.}\e. Secondly, if $k$ is an {\it arbitrary} field and $\nu$ is as above, Brion \cite{bri15} established the surjectivity of the pullback morphism $\nu^{\le *}\colon \picxk\to \picxkn$ and described its kernel for any projective\,\footnote{A weaker hypothesis suffices but we assume projectivity here to simplify the exposition.} and geometrically integral $k$-scheme $X$ with a {\it finite} non\le-\le normal locus. In fact, Brion worked in a much more general setting that we now (partially) describe over a field. Let $\widetilde{X}$ be a projective and geometrically integral $k$-scheme, let $\widetilde{Y}$ be a {\it finite} closed subscheme of $\widetilde{X}$ with corresponding inclusion morphism $\widetilde{\iota}\colon \widetilde{Y}\to \widetilde{X}$ and let 
$\psi\colon \widetilde{Y} \to Y$ be a schematically dominant morphism between finite $k$-schemes. By \cite[\S5]{fer}, there exists a pushout $X$ of $(\e\widetilde{\iota}\le,\lbe\psi)$ in the category of $k$-schemes, i.e., a cocartesian diagram
\begin{equation}\label{pin0}
\xymatrix{\widetilde{Y}\,\ar@{^{(}->}[r]^{\widetilde{\iota}}\ar[d]_{\psi}&\widetilde{X}\ar[d]_{\varphi}\\
Y\,\ar@{^{(}->}[r]^(.45){\iota}&X,
}
\end{equation}
where $\iota$ is a closed immersion, $\varphi$ is finite and the morphism $\widetilde{X} \setminus \widetilde{Y} \to X \setminus Y$ (induced by $\varphi$) is an isomorphism. The $k\le$-\le scheme $X$ is projective and geometrically integral and the above diagram is cartesian as well. We say that \emph{$X$ is obtained by pinching $\widetilde{X}$ along $\widetilde{Y}$ via $\psi$} and that \eqref{pin0} is the {\it pinching diagram} associated with $(\lbe\widetilde{X}, \widetilde{Y},\psi\le)$. Conversely, {\it every} projective and geometrically integral $k$-scheme $X$ with a finite non-normal locus can be obtained via the above construction. Indeed, if $Y$ is the closed subscheme of $X$ defined by the ideal ${\rm Ann}_{\e\cO_{\lbe X}}\be(\nu_{*}\cO_{\be X^{\lle\be N}}\!/\cO_{\lbe X}\be)$, then $X$ can be recovered by pinching $\widetilde{X}=X^{\lbe N}$ along $\widetilde{Y}\defeq Y\times_{X}\widetilde{X}$ via the first projection $\psi\colon \widetilde{Y}\to Y$. See \cite[\S3.1]{lau}. Now, by Brion's key insight \cite[Lemma 2.2]{bri15}, the pinching diagram \eqref{pin0} induces an exact sequence of locally algebraic $k$-group schemes
\begin{equation}\label{dos0}
0\to \mu^{\widetilde{Y}\!/\le Y}\to\picxk\overset{\varphi^{*}}{\to} \picxkt\to 0,
\end{equation}
where $\mu^{\widetilde{Y}\!/\le Y}$ is smooth, connected, affine and algebraic. See \cite[Remark 2.6]{bri15}. In this paper we combine, in a rather elaborate way, the sequences \eqref{brpic} and \eqref{dos0} to compute the kernel of the induced pullback map 
$\varphi^{\le *}_{1}\be\colon\be \br_{\lbe 1} X\to \br_{\lbe 1} \widetilde{X}$ {\it solely} in terms of the Brauer groups of the residue fields of $\e Y$ and $\widetilde{Y}$ and the Amitsur subgroups of $X$ and $\widetilde{X}$ in $\br\e k$, i.e., the images of the maps $(\pic\le\xs\le)^{\lbe\g}\to\br\e k$ and $(\pic\le\widetilde{X}^{\rm s}\le)^{\lbe\g}\to\br\e k$ in the sequence \eqref{brpic} for $X$ and $\widetilde{X}$. Further, we bound 
(and, in some cases, prove the triviality of\e) $\cok\varphi^{\le *}_{1}$ in terms of Brauer groups of fields and analogs, in cohomological degree $3$, of relative Brauer groups of finite field extensions. In order to state our results, we introduce the following definitions. If $f\colon X\to S$ is a morphism of schemes, we write $\bxs$ for the kernel of the pullback map $\br S\to\br X$ induced by $f$. If $S=\spec k$, where $k$ is a field, $\bxs$ will be denoted by $\bxk$. We have $\bxk=\krn\be[\, f^{*}\lbe\colon \br\le k\to \brok\e]$, which is an abelian group of finite exponent annihilated by the {\it index} $I(\be X\lle)$ of $X$ over $k$, i.e., the least positive degree of a $0\le$-\le cycle on $X$. If $L/k$ is a field extension, then $\br(\spec L/\lle k)=\br(\lbe L\lbe/k\le)=\krn\be[\e\br k\lbe\to\lbe\br L\le]$ is the {\it relative Brauer group of $L/k$}. If $X$ is $k\e$-\le proper and geometrically integral then, by the exactness of \eqref{brpic}, $\bxk$ agrees with the Amitsur subgroup of $X$ in $\br\e k$. Now set $\bra=\cok\be[\, f^{*}\lbe\colon \br\le k\to\brok\e]$. Then \eqref{brpic} induces isomorphims of abelian groups
\begin{equation}\label{tri0}
\bxk=\cok\!\!\left[\e\pic\le X\hookrightarrow (\pic\le\xs\le)^{\lbe\g}\e\right]
\end{equation}
and
\[
\hskip 1.65cm\bra=\krn\be[H^{\le 1}_{\et}\lbe(k,\picxk\be)\overset{\!\!d}{\to} H^{\le\lle 3}_{\et}\lbe(k,\G_{m}\lbe)],
\]
where $d$ is the differential in \eqref{brpic}. We will also need the following analog of $\bxk$ in cohomological degree $\le 3$:
\[
H^{\le\lle 3}\lbe(\lbe X/\lbe k\e)=\krn\be[H^{\e 3}_{\et}\lbe( k,\G_{\le m}\lbe)\to H^{\e 3}_{\et}\lbe(X,\G_{\le m}\lbe)].
\]
If $\varphi\colon\widetilde{X}\to X$ is an arbitrary morphism of $k$-schemes,  $\varphi_{\lbe\rm a}^{*}\colon \bra\to\brap$ will denote the natural map induced by $\varphi^{\le *}_{1}\be\colon\be \br_{\lbe 1} X\to \br_{\lbe 1} \widetilde{X}$. Note that
\[
\bxk=\krn\be[\e\br\le k\overset{\!f^{\le *}}{\to} \brok\e]\subseteq\krn\be[\e\br\le k\overset{\varphi_{1}^{ *}\lbe\circ\lbe f^{\lle *}}{\lra} \br_{\lbe 1} \widetilde{X}\e]=\bxkp\subseteq\br k.
\]
The subquotient $\bxkp/\e\bxk$ of $\br\e k$ intervenes frequently in our considerations. The main results of this paper are summarized in the following statement.

\begin{theorem}\label{thm a} Let $k$ be an arbitrary field and let $(\lbe\widetilde{X}, \widetilde{Y},\psi\le)$ be as above, i.e., $\widetilde{X}$ is a projective and geometrically integral $k$-scheme, $\widetilde{Y}$ is a finite closed subscheme of $\widetilde{X}$ and $\psi\colon \widetilde{Y} \to Y$ is a schematically dominant morphism between finite $k$-schemes. If $X$ is obtained by pinching $\widetilde{X}$ along $\widetilde{Y}$ via $\psi$, then the pinching diagram \eqref{pin0} induces
\begin{enumerate}
\item[(i)] an injection of abelian groups
\[
\frac{\bxkp}{\bxk}\into \prod_{y\e\in\e Y}\bigcap_{\e \widetilde{y}\e\in\e \widetilde{Y}_{y}}\!\br\lbe(\kappa(\widetilde{y}\,)/\kappa(y)),
\]
\item[(ii)] a canonical exact sequence of abelian groups 
\[
\hskip .2cm 0\to\cok\!\!\!\left[\le \frac{\bxkp}{\bxk}\into \prod_{y\e\in\e Y}\!\bigcap_{\e \widetilde{y}\e\in\e \widetilde{Y}_{y}}\!\br\lbe(\kappa(\widetilde{y}\,)/\kappa(y))\right]\!\to \bra\overset{\!\!\!\varphi_{\lbe\rm a}^{*}}{\to}\brap\to H^{\le 2}_{\et}\lbe(k,\mu^{\widetilde{Y}\!/\le Y}),
\]
where $H^{\le 2}_{\et}\lbe(k,\mu^{\widetilde{Y}\!/\le Y})$ is an extension
\[
\hskip .21cm 0\to\!\!\prod_{\e y\e\in\e Y}\!\be\cok\!\!\!\left[\lbe\br\le \kappa(y)\!\to\!\!\be\prod_{\e\widetilde{y}\e\in\e \widetilde{Y}_{y}}\!\br\e \kappa(\le\widetilde{y}\,)\be\right]\!\to\! H^{\le 2}_{\et}\lbe(k,\mu^{\widetilde{Y}\!/\le Y})\!\to\!\prod_{y\e\in\e Y}\!\bigcap_{\e \widetilde{y}\e\in\e \widetilde{Y}_{y}}\!H^{\le\lle 3}\lbe(\kappa(\le\widetilde{y}\,)/\lbe \kappa(y))\to 0,
\]
and
\item[(iii)] an exact sequence of abelian groups 
\[
0\to \krn \varphi_{1}^{*}\to \brok\!\overset{\!\!\varphi_{\lbe\lle 1}^{*}}{\lra}\!\brokp\!\to\! \cok\varphi_{\lbe\rm a}^{*}\to 0,
\]
where $\krn \varphi_{1}^{*}$ is an extension
\[
\hskip .5cm 0\to\frac{\bxkp}{\bxk}\to\krn \varphi_{1}^{*}\to \cok\!\!\!\left[\le \frac{\bxkp}{\bxk}\into \prod_{y\e\in\e Y}\!\bigcap_{\e \widetilde{y}\e\in\e \widetilde{Y}_{y}}\!\br\lbe(\kappa(\widetilde{y}\,)/\kappa(y))\right]\to 0.
\] 
\end{enumerate}
\end{theorem}

\smallskip

The group $\prod_{\e y\e\in\e Y}\be\bigcap_{\, \widetilde{y}\e\in\e \widetilde{Y}_{y}}\be\br\lbe(\kappa(\widetilde{y}\,)/\kappa(y))$ which appears in the above statement, and therefore also $\bxkp/\e\bxk$, is an abelian group of finite exponent annihilated by $\lcm\{\le I\lbe\big(\le\widetilde{Y}_{ y}\lbe\lle\big)\colon\be y\in Y\lle\}$. See Remark \ref{alc}\,. Here $\widetilde{Y}_{ y}=\psi^{-1}\lbe(\lle y\lle)$ is the fiber of $\psi$ at $y\in Y$.

\smallskip

\begin{corollary} \label{cor} Let $k$ be any field and let $X$ be a projective and geometrically integral $k$-scheme with a finite non\le-\le normal locus $Y$ such that $X^{\be N}$ is $k\le$-\e isomorphic to $\Pn_{\! k}^{\e\le{\rm dim}\e X}\be$. Then there exists a canonical isomorphism of abelian groups
\[
\brok=\br\e k\oplus \prod_{y\e\in\e Y}\bigcap_{\e \widetilde{y}\e\in\e \widetilde{Y}_{y}}\!\br\lbe(\kappa(\widetilde{y}\,)/\kappa(\le y)),
\]  
where $\widetilde{Y}$ is the ramification locus of $X^{\be N}\to X$.
\end{corollary}

See Example \ref{ex1} for a more general statement.

\medskip

The {\it proof} of Theorem \ref{thm a} also yields the following generalization of a theorem established by Ballico and Koll\'ar \cite{bk} when $S$ is the spectrum of a field of characteristic zero.

\begin{theorem}\label{thm B} Let $\widetilde{X}$ be a reduced scheme, let $S$ be a connected locally noetherian scheme and let $\widetilde{X}\to S$ be a projective and flat morphism with integral geometric fibers. Further, let $\widetilde{Y}$ be a closed subscheme of $\widetilde{X}$ which is finite and faithfully flat over $S$ and let $\psi\colon\widetilde{Y} \to Y$ be a schematically dominant morphism between finite and faithfully flat $S$-schemes such that the pullback map $\psi^{\le *}\colon\pic\e Y\onto \pic\e \widetilde{Y}$ is {\rm surjective}. If $X$ is obtained by pinching $\widetilde{X}$ along $\widetilde{Y}$ via $\psi$, then
\[
\bxs=\bxsp\cap\bys,
\]
where the intersection takes place in $\br\le S$.
\end{theorem}

If $S=\spec k$, where $k$ is a field, the preceding theorem describes the Amitsur subgroup in $\br\e k$ of the pinched variety $X$ in terms of the corresponding group for $\widetilde{X}$ and the group
$\br(\le Y\!/\le k\lle)=\bigcap_{\, y\e\in\e Y}\be\br\lbe(\kappa(\lle y)\lbe/k\lle)$ \eqref{bk0}.
 
\medskip

If $k$ is a local field, $\br\e k$ is isomorphic to a subgroup of $\Q/\Z$. Consequently, if $X$ is $k\e$-\e proper and geometrically integral, then $\bxk\simeq \cok\!\!\left[\e\pic\le X\hookrightarrow (\pic\le\xs\le)^{\lbe\g}\e\right]$ \eqref{tri0} is a finite subgroup of $\br\e k$. If $X$ is smooth and $k$ is a $p\e$-adic field (i.e., a finite extension of $\Q_{\lle p}$), Roquette and Lichtenbaum showed that the order of $\bxk$ is equal to the index of $X$ over $k$. See \cite[Theorem 1]{roq} and \cite[Theorem, p.~120]{lich}. In section \ref{cin} we apply Theorem \ref{thm B} above with $S=\spec k$, where $k$ is any local field, and extend the Roquette\e-\lle Lichtenbaum theorem to a class of (not necessarily smooth) proper and geometrically integral $k\e$-\le curves that contains {\it all} proper and geometrically integral $k\e$-\le curves if $k$ is a $p\e$-\le adic field. The precise statement is the following.

\begin{corollary} Let $X$ be a proper and geometrically integral curve over a local field $k$. If $X^{\be N}$ is {\rm smooth}, then 
\[
\left[\lle(\pic\le\xs\le)^{\lbe\g}\lbe\colon\be \pic\le X \e\right]=I(\be X\lle),
\]
where $I(\be X\lle)$ is the index of $X$ over $k$.
\end{corollary}

If $k$ is a local function field, i.e., a finite extension of $\mathbb{F}_{\be p}((t))$, then the above corollary applies, in particular, to (proper and geometrically integral) local complete intersection $k\e$-\le curves whose Jacobian numbers are less than $p$. See \cite[Theorem 4.10]{iil20}. Regarding the above corollary, we call the reader's attention to the fact that, over an imperfect field such as $\mathbb{F}_{\be p}((t))$, there exist many curves $X$ such that the normal (i.e., regular) curve $X^{\be N}$ is {\it not} smooth. See Examples \ref{ex3} and \ref{las} below.

\section*{Acknowledgements} I thank Otto Overkamp for helpful comments regarding the proof of Proposition \ref{sn} and Michel Brion for providing Lemma \ref{bri}\,.

\section{Preliminaries}\label{pre}

We say that a scheme $X$ is an {\it FA scheme} if
every finite subset of $X$ is contained in an affine open subset of $X$. If $X$ admits an ample invertible sheaf, then $X$ is an FA scheme \cite[II, Corollary 4.5.4]{ega}. In particular, every proper scheme of dimension $1$ over a field is an FA scheme \cite[Lemmas 0A26 and 0B45]{sp}.

\smallskip

If $S$ is a scheme, $\ses$ will denote the category of abelian sheaves on the small \'etale site over $S$. If $f\colon X\to S$ is a morphism of schemes and $r\geq 0$ is an integer, we will write
\begin{equation}\label{pbm}
H^{\le\lle r}\be(\be X\be/\be S\le)\defeq \krn\be[H^{\le r}_{\et}\be( S,\G_{\le m}\lbe)\overset{f^{*}}{\to} H^{\le r}_{\et}\be( X,\G_{\le m}\lbe)],
\end{equation}
where $f^{\lle *}$ is the pullback homomorphism induced by $f$. If $f\colon \spec K\to\spec k$ is the canonical morphism induced by a field extension $K/k$, then the abelian group \eqref{pbm} will be denoted by $H^{\le r}\be(\be K\!/k)$.

\smallskip

We henceforth assume that $S$ is locally noetherian, so that a morphism $q\colon T\to S$ is finite and locally free if, and only if, $q$ is finite and flat \cite[02K9, Lemma 29.48.2]{sp}. A {\it finite and flat quasi-section} of $f$ is a finite and faithfully flat morphism $q\colon T\to S$ such that $q=f\circ h$ for some morphism $h\colon T\to X$. If this is the case, then the induced morphism $(h, 1_{T})_{S}\colon T \to X_{T}$ is a section of $f_{T}\colon X_{T}\to T$. If $f$ admits a finite and flat quasi-section of constant degree, then the {\it index $I(\lbe X\be/\lbe S\e)$ of $f$} is the greatest common divisor of the degrees of all finite and flat quasi-sections of $f$ of constant degree. Since the proof of \cite[Proposition 3.8.1, p.~96]{ctsk} is valid in arbitrary cohomological degree, \cite[Remark 3.1]{ga17} remains valid if the word {\it \'etale} in [loc.cit.] is replaced with the word {\it flat}. Consequently, if $I(\lbe X\be/\lbe S\e)$ is defined, then
\begin{equation}\label{ihr}
I(\lbe X\be/\lbe S\e)\be\cdot\be H^{\le\lle r}\lbe(\lbe X\be/\be S\le\lle)=0 
\end{equation}
for every $r\geq 1$ by (the flat analog of) \cite[Remark 3.1(d)]{ga17}. If $\widetilde{f}\colon \widetilde{X}\to S$ is another morphism of schemes and $\varphi\colon \widetilde{X}\to X$ is an $S$-morphism, then a finite and flat quasi-section of $\widetilde{f}$ is also a finite and flat quasi-section of $f$. Consequently, if $I(\lbe \widetilde{X}\be/\lbe S\e)$ is defined, then $I(\lbe X\!/\lbe S\e)$ is also defined and divides $I(\lbe \widetilde{X}\be/\lbe S\e)$. We also note that, if $f$ is finite and faithfully flat of constant degree, then $I(\lbe X\!/\lbe S\le)$ is defined since we may take $q=f$ and $h={\rm Id}_{X}$ above.

If $X$ is a scheme over a field $k$, we will write $I(\lbe X\lle)$ for $I\lbe\lle(\lbe X\be/\spec k\le)$. If $X$ is algebraic, i.e., of finite type over $k$, then $I(\be X\lle)$ is the greatest common divisor of the degrees $[\le\kappa(x)\colon\! k\e]$ of all closed points $x$ of $X$. Equivalently, $I(\lbe X\lle)$ is the least positive degree of a $0\e$-\le cycle on $X$. If $K/k$ is a finite field extension and $X=\spec K$, then \eqref{ihr} yields
\begin{equation}\label{ihrk}
[\le K\lbe\colon\! k\e]\be\cdot\be H^{\le r}\be(\lbe K\!/k)=0
\end{equation}
for every $r\geq 1$.

\smallskip

We now recall that the (\'etale) {\it relative Picard functor of $X$ over $S$} is the \'etale sheaf $\picxs$ on $S$ associated to the presheaf $(\textrm{Sch}/S\e)\to \mathbf{Ab}, (T\to S\le)\mapsto\pic X_{T}$. We have
$\picxs=R^{\le 1}_{\et}\le f_{\lbe *}\G_{m,X}$ by \cite[\S2]{klei}. Now, for any scheme $Y$, let $\br\e Y=H^{2}(Y_{\et},\G_{\le m}\lbe)$ be the cohomological Brauer group of $Y$ and let $\brxs=R^{\le 2}_{\et}\le f_{\lbe *}\G_{m,\le X}$ be the sheaf on $S_{\et}$ associated to the presheaf $(T\to S)\mapsto \br X_{T}$. Set
\begin{equation}\label{br1}
\bro=\krn\!\!\left[\br X\be\to\be\brxs(S\e)\e\right]\!,
\end{equation}
where the indicated map is an instance of the canonical adjoint homomorphism $P(S\le)\to P^{\#}\be(S\le)$, where $P$ is a presheaf of abelian groups on $(\textrm{Sch}/S\le)$ and $P^{\le\#}$ is its associated (\'etale) sheaf \cite[Remark, p.~46]{t}. The pullback map $f^{\le *}\colon\br\le S\to \br X$ factors through $\bro\subseteq\br X$ \cite[p.~2754]{ga18}. Set 
\begin{equation}\label{relb}
\bxs=\krn\!\!\left[\e\br\le S\to \bro\le\right]
\end{equation}
and
\begin{equation}\label{bra}
\hskip 1.11cm\bra=\cok\!\be\left[\e\lle\br\le S\to\bro\e\right]\!.
\end{equation}
Since $\bxs=H^{\le\lle 2}\lbe(\lbe X\be/\be S\e)$ by \eqref{pbm}, the formula \eqref{ihr} yields $I(\lbe X\be/\lbe S\e)\be\cdot\be \bxs=0$ when $I(\lbe X\be/\lbe S\e)$ is defined.

\smallskip

\begin{remarks}\label{int} Let $k$ be a field.
\begin{enumerate}
\item[(a)] If $f\colon X\to\spec k$ is quasi-compact and quasi-separated, then $\picxk(k\e)=(\pic\xs\le)^{\lbe\g}$ and $\br_{\be X\be/k}\lbe(k\e)=(\br\xs\le)^{\lbe\g}$ by \cite[II, Corollary 2.2(ii), p.~94, and Theorem 6.4.1, p.~128]{t}. Therefore $\bro=\krn\!\be\left[\e\br X\be\to\be\br\xs\e\right]$ \eqref{br1} is the algebraic Brauer group of $X$ over $k$.
\item[(b)] If $X$ is a proper algebraic curve over $k$, then $\br\xs=0$ by \cite[Corollary 5.8, p.~132]{gb}, whence $\bro=\br X$.
\item[(c)] By (b) and \cite[Theorem 5.6.1(vii), p.~148]{ctsk}, $\br_{\be 1}\lle\Pn^{\le 1}_{\!k}=\br\e\Pn^{\le 1}_{\!k}=\br\e k$.
\end{enumerate}
\end{remarks}

\begin{lemma}\label{ker-cok} If $\mathcal A$ is an abelian category and $f$ and $g$ are morphisms in $\mathcal A$ such that $g\be\circ\!\be f$ is defined, then there exists a canonical exact sequence in $\mathcal A$ 
\[
0\to\krn f\to\krn\lbe(\e\lle g\be\circ\!\be f\e)\to\krn g\to\cok f\to\cok\be(\e\lle g\be\circ\!\be f\e)\to\cok g\to 0.
\]
\end{lemma}
\begin{proof} See, for example, \cite[1.2]{bey}.
\end{proof}

If $k$ is a field, a {\it $k$-split unipotent $k$-group} is an algebraic $k$-group that admits a composition series whose successive quotients are each 
$k$-isomorphic to $\G_{ a,\le k}$.

\begin{lemma}\label{spl} Let $k$ be a field and let $U$ be a commutative, smooth, connected and unipotent algebraic $k$-group.
\begin{enumerate}
\item[(i)] If $U$ is $k$-split, then $H^{\le r}_{\le\et}\lbe(k,U\le)=0$ for every $r\geq 1$.
\item[(ii)] $H^{\le r}_{\le\et}\lbe(k,U\le)=0$ for every $r\geq 2$. 
	\end{enumerate}
\end{lemma}
\begin{proof} Assertion (i) follows from the case $U=\G_{a,\le k}$ \cite[X, Proposition 1, p.~150]{ser2} by induction on the length of a composition series for $U$. In (ii), if $\car\le k=0$ (or, more generally, if $k$ is perfect), then $U$ is $k$-split by \cite[Corollary 15.5(ii), p.~205]{bo}, whence (ii) follows from (i). Thus in (ii) we may assume that $\car\le k=p>0$. Since $U$ is annihilated by some power of $p$ \cite[XVII, Theorem 3.5]{sga3} and the $p\e$-\e cohomological dimension of ${\rm Gal}\le(\ksep\be/k)$ is $\leq 1$ \cite[II, \S2.2, Proposition 3, p.~86]{scg}, we have $H^{\le r}_{\le\et}\lbe(k,U\le)=H^{\lle r}\lbe({\rm Gal}(\ksep\be/k),U(\ksep)\le)=0$ for every $r\geq 2$, as asserted. 
\end{proof}

The following lemma and its proof were provided by Michel Brion.

\begin{lemma}\label{bri} Let $k$ be a field and let $X$ be a geometrically integral $k\le$-\le scheme. Then the normalization $X^{\lle\be N}$ of $X$ is geometrically integral.
\end{lemma}
\begin{proof} We may assume that $X$ is affine. Then $X^{\lle\be N}=\spec R$ is also
affine. Since $R$ is a subring of the function field $k(X^{\lle\be N})$, which coincides with $k(X)$ by \cite[035E, Lemma 0BXC]{sp}, $R\otimes_{k}L$ is a subring of $k(X)\otimes_{k}L$ for any field extension $L/k$. Since $k(X)\otimes_{k}L$ is integral, the lemma follows.	
\end{proof}

\section{Pinched schemes}\label{piag}

Recall that $S$ is a locally noetherian scheme. Let $\widetilde{X}$ be a reduced FA scheme, let $\widetilde{f}\colon \widetilde{X}\to S$ be a proper and flat morphism with integral geometric fibers and let $\widetilde{Y}$ be a closed subscheme of $\widetilde{X}$ which is finite, faithfully flat and of constant degree over $S$. Further, let $Y$ be a finite and faithfully flat $S$-scheme and let $\psi\colon \widetilde{Y} \to Y$ be a schematically dominant morphism of $S$-schemes, i.e, the induced morphism of Zariski sheaves $\cO_{\le Y}\to\psi_{*}\cO_{\le\widetilde{Y} }$ is {\it injective} \cite[Proposition 5.4.1 and Definition 5.4.2, pp.~283-284]{ega1}. By \cite[\S2.1]{bri15}, there exists a cartesian and cocartesian diagram of $S$-schemes
\begin{equation}\label{pin}
\xymatrix{\widetilde{Y}\,\ar@{^{(}->}[r]^{\widetilde{\iota}}\ar[d]_{\psi}&\widetilde{X}\ar[d]_{\varphi}\\
Y\,\ar@{^{(}->}[r]^(.45){\iota}&X,
}
\end{equation}
where: $\widetilde{\iota}$ is the inclusion, $\iota$ is a closed immersion, $\varphi$ is finite and surjective, the morphism $\widetilde{X} \setminus \widetilde{Y} \to X \setminus Y$ induced by $\varphi$ is an isomorphism, $X$ is an FA scheme and the structural morphism $f\colon X\to S$ is proper and flat with integral geometric fibers. See \cite[comment after the proof of Lemma 2.1]{bri15}. We say that 
\emph{$X$ is obtained by pinching $\widetilde{X}$ along $\widetilde{Y}$ via $\psi$}.
Note that, since $I(\e \widetilde{Y}\!/\lbe S\e)$ is defined, $I(\lbe \widetilde{X}\be/\lbe S\e)$, $I(\e Y\be/\lbe S\e)$ and $I(\lbe X\be/\lbe S\e)$ are also defined and the following holds
\begin{equation}\label{divv}
I(\lbe X\be/\lbe S\e)|\gcd\{ I(\lbe \widetilde{X}\be/\lbe S\e),I(\e Y\be/\lbe S\e) \}.
\end{equation}
The pinching diagram \eqref{pin} induces the right-hand square of the following commutative diagram of abelian groups
\[
\xymatrix{\br\e S\ar@{=}[d]\ar[r]^{f^{*}}&\br\le X\ar[d]^(.4){\iota^{\lbe *}}\ar[r]^{\varphi^{*}}&\br\e \widetilde{X}\ar[d]^(.4){\widetilde{\iota}^{\,*}}\\
\br\e S\ar[r]^{(\le f\circ\e \iota\le)^{*}}&\br\le Y\ar[r]^{\psi^{\le *}}&\br\e \widetilde{Y}.
}
\]
Applying Lemma \ref{ker-cok} to the rows of the above diagram, we obtain an exact and commutative diagram of abelian groups
\begin{equation}\label{zwe}
\xymatrix{0\ar[r]&\bxs\ar[d]^{\subseteq}\ar[r]^{\subseteq}&\bxsp\ar[d]^{\subseteq}\ar[r]^{f^{\lle *}}&\br\lbe(\be\widetilde{X}\!/\be X\lle)\ar[d]^(.4){\iota^{\lbe *}}\\
0\ar[r]&\bys\ar[r]^{\subseteq}&\bysp\ar[r]^{(\le f\circ\e \iota\le)^{*}}&\br\lbe(\e\widetilde{Y}\!/Y\le),
}
\end{equation}
where all groups are instances of \eqref{relb} and all arrows on the left-hand square are inclusion maps between subgroups of $\br\e S$.

In the next statement, $R_{\e Y\be /S}$ (respectively, $R_{\, \widetilde{Y}\be/S}$) denotes the Weil restriction functor associated to the finite and locally free morphisms $Y\to S$ (respectively, $\widetilde{Y}\to S\e$). See \cite[\S7.6]{blr}.

\begin{theorem}\label{ex}{\rm (Brion\lle)} The pinching diagram \eqref{pin} induces an exact sequence in $\ses$
\[
0\to R_{\e Y\be /S}\lle(\G_{m, Y}\be)\overset{\psi^{\lle *}}{\to} R_{\, \widetilde{Y}\be/S}(\G_{m, \widetilde{Y}}\lbe)\to\picxs\overset{\varphi^{*}}{\to} \picxst\to 0.
\]
\end{theorem}
\begin{proof} See \cite[Corollary 2.3]{bri15}.
\end{proof}

The exact sequence of the theorem induces the following exact sequences in $\ses$: 
\begin{equation}\label{uno}
0\to R_{\e Y\be/S}(\G_{m,\lle Y})\overset{\psi^{*}}{\to} R_{\e \widetilde{Y}\be/S}(\G_{m,\lle  \widetilde{Y}}\lbe)\to\mu^{\widetilde{Y}\!/\le Y}\to 0
\end{equation}
and
\begin{equation}\label{dos}
0\to \mu^{\widetilde{Y}\!/\le Y}\to\picxs\overset{\varphi^{*}}{\to} \picxst\to 0,
\end{equation}
where 
\begin{equation}\label{myy}
\mu^{\widetilde{Y}\!/\le Y}\defeq R_{\e \widetilde{Y}\be/S}(\G_{m\lle,\lle \widetilde{Y}}\lbe)/\psi^{*}\lbe(R_{\e Y/S}(\G_{m\lle, Y}))\in\ses.
\end{equation}
If $(\e Y,\widetilde{Y}\e)=(\e\spec A\lle,\le\spec B\le)$, we will write $\mu^{B/\lbe A}$ for $\mu^{\widetilde{Y}\!/\le Y}\!$.

By \cite[Theorem 6.4.2(ii), p.~128]{t}, $H^{\le r}_{\et}\lbe(S,R_{\e Y\!/\lbe S}(\G_{m,\lle Y}\lbe))=H^{\le r}_{\et}\lbe(\le Y,\G_{\le m}\lbe)$
for every integer $r\geq 0$, and similarly for $\widetilde{Y}$. Thus \eqref{uno} and \eqref{dos} induce the following exact sequences of abelian groups:
\begin{equation}\label{seq1}
\begin{array}{rcl}
0&\to& \pic\e \widetilde{Y}\!/\le \psi^{\lle *}\lbe(\lle\pic\e Y\lle)\to H^{\le 1}_{\et}\lbe(S,\mu^{\widetilde{Y}\!/\le Y})\overset{\!\partial^{\le 1}}{\to} \br\e Y\overset{\!\psi^{\lle *}}{\to}\br\e \widetilde{Y}\\
&\to& H^{\le 2}_{\et}\lbe\lle(S,\mu^{\widetilde{Y}\!/\le Y})\to H^{\le 3}_{\et}\lbe(\e Y,\G_{\le m}\lbe)\overset{\!\psi^{\lle *}}{\to} H^{\le 3}_{\et}\lbe(\e\widetilde{Y},\G_{\le m}\lbe)\to\dots
\end{array}
\end{equation}
and
\begin{equation}\label{seq2}
\begin{array}{rcl}
\picxs(\lbe S\le)\to\picxst(\lbe S\le)&\overset{\!\partial^{\e 0}}{\to}& H^{\le 1}_{\et}\lbe(S,\mu^{\widetilde{Y}\!/\le Y})\to H^{\le 1}_{\et}\lbe(S,\picxs)\\
&\overset{\!\varphi^{\lle *}}{\to}& H^{\le 1}_{\et}\lbe(S,\picxst)\to H^{\le 2}_{\et}\lbe(S,\mu^{\widetilde{Y}\!/\le Y}),
\end{array}
\end{equation}
where the maps $\partial^{\le 1}$ and $\partial^{\e 0}$ are connecting homomorphisms in \'etale cohomology.

Next, since $\widetilde{f}\colon \widetilde{X}\to S$ is proper and flat with integral geometric fibers, the canonical map $\cO_{\le T}\to (\widetilde{f}_{\lle T}\lbe)_{\lbe *}\cO_{\lbe \widetilde{X}_{T}}$ is an isomorphism of Zariski sheaves on $T$ for every $S$-scheme $T$ \cite[Exercise 9.3.11, pp.~260 and 303]{klei}. Consequently, the Cartan-Leray spectral sequence
$H^{\le r}_{\et}\lbe(\lbe S, R^{\le s}\be\lle \widetilde{f}_{\be *}\G_{m,\le \widetilde{X}}\be)\Rightarrow H^{\le r+s}_{\et}(\lbe \widetilde{X}, \G_{\le m}\lbe)$ induces an exact sequence of abelian groups
\begin{equation}\label{seq3}
0\to\pic S\to \pic \widetilde{X}\to \picxst(\lbe S\le)\overset{\!d_{\lbe 2,\lle \widetilde{X}}^{\e 0,\le 1}}{\to}\br\le S\to\brop\to H^{\le 1}_{\et}\lbe(S,\picxst)\to H^{\le\lle 3}\lbe(\lbe  \widetilde{X}\be/\lbe S\le), 
\end{equation}
where $H^{\le\lle 3}\lbe(\lbe  \widetilde{X}\be/\lbe S\le)$ and $\brop$ are the groups \eqref{pbm} and \eqref{br1}, respectively (cf. \cite[Proposition 5.4.2, p.~138]{ctsk}).

When $S=\spec k$, where $k$ is a field, we have $\pic S=0$ and $\picxkt(\lbe k\le)=(\pic\le\widetilde{X}^{\rm s}\le)^{\lbe\g}$ by Remark \eqref{int}(a), whence the third map in \eqref{seq3} is an injection $\pic \widetilde{X}\hookrightarrow(\pic\le\widetilde{X}^{\rm s}\le)^{\lbe\g}$. Identifying $\pic \widetilde{X}$ with its image in $(\pic\le\widetilde{X}^{\rm s}\le)^{\lbe\g}\!$ under the preceding injection, the map $d_{\lbe 2,\lle \widetilde{X}}^{\e\lle 0,\le 1}$ in \eqref{seq3} for $S=\spec k$ induces an isomorphism of abelian groups
\begin{equation}\label{cbx}
(\pic\le\widetilde{X}^{\rm s}\le)^{\lbe\g}\!/\e\pic\le \widetilde{X}\isoto \bxkp.
\end{equation}

\begin{lemma}\label{comm} The following diagram commutes (up to sign)
\begin{equation}\label{cod}
\xymatrix{\picxst(\lbe S\le)\ar[d]^{\overline{d}_{\lle 2,\lle \widetilde{X}}^{\, 0,\le 1}}\ar[r]^(.45){\partial^{\le\lle 0}}&H^{\le 1}_{\et}\lbe(S,\mu^{\widetilde{Y}\!/\le Y})\ar@{->>}[r]^(.57){\overline{\partial}^{\e 1}}&\br\lbe(\e\widetilde{Y}\!/Y\le)\ar@{=}[d]\\
\bxsp\ar[r]^{\subseteq}&\bysp \ar[r]^{(\le f\circ\e \iota\le)^{*}}&\br\lbe(\e\widetilde{Y}\!/Y\le),	
}
\end{equation}
where $\overline{d}_{\le 2,\lle \widetilde{X}}^{\, 0,\le 1}$ is induced by the first differential in \eqref{seq3}, the bottom row comes from diagram \eqref{zwe}, $\partial^{\e 0}$ is the second map in \eqref{seq2} and $\overline{\partial}^{\e 1}$ is induced by the third map in \eqref{seq1}.
\end{lemma}
\begin{proof} Since $\widetilde{X}$ is reduced, there exists a canonical exact sequence in $\widetilde{X}_{\et}^{\sim}$
\begin{equation}\label{res}
0\to\G_{\lbe m,\le \widetilde{X}}\to {\ms R}_{\! \widetilde{X}}^{\e *}\to\msdiv_{\lbe \widetilde{X}}\to 0,
\end{equation}
where ${\ms R}_{\!\widetilde{X}}^{\e *}$ (respectively, $\msdiv_{\lbe \widetilde{X}}$) is the \'etale sheaf of invertible rational functions (respectively, Cartier divisors) on $\widetilde{X}$ \cite[p.~71]{gb}. Since $\widetilde{f}_{\be *}\G_{ m,\le \widetilde{X}}=\G_{\lbe m,\le S}$ and $R^{1}\be \widetilde{f}_{\be *}\e{\ms
R}_{\! \widetilde{X}}^{\e *}=0$ by \cite[Lemma 1.6]{gb} and \cite[Proposition III.1.13, p.~88]{miet}, the sequence \eqref{res} induces an exact sequence in $\ses$
\begin{equation}\label{res2}
0\to\G_{m,\le S}\to \widetilde{f}_{\be *}\le{\ms R}_{\! \widetilde{X}}^{\e *}\overset{\!\widetilde{q}}\to \widetilde{f}_{\be *}\msdiv_{\lbe \widetilde{X}}\to \picxst\to 0,
\end{equation}
which gives rise to two short exact sequences
\begin{equation}\label{res3}
0\to\G_{m,\le S}\to \widetilde{f}_{\be *}\le{\ms R}_{\! \widetilde{X}}^{\e *}\to\img \e\widetilde{q}\to 0
\end{equation}
and
\begin{equation}\label{res4}
0\to\img\e \widetilde{q}\to \widetilde{f}_{\be *}\msdiv_{\lbe \widetilde{X}}\to \picxst\to 0.
\end{equation}
The sequences \eqref{res2}, \eqref{res3} and \eqref{res4} are of the form \cite[(1.A.8), (1.A.9) and (1.A.10) (respectively), pp.~398\le-\le 399]{cts87}. Thus, by the commutativity (up to sign) of \cite[diagram (1.A.11), p.~399]{cts87}, the differential 
$d_{\lbe 2,\lle \widetilde{X}}^{\, 0,\le 1}\colon\picxst(\lbe S\e)\to\br\le S$ in the sequence \eqref{seq3} is (up to sign) the composition
\[
\picxst(\lbe S\e)\overset{\delta^{\lle 0}}{\to} H^{\le 1}_{\et}\lbe(S,\img \e\widetilde{q}\le\,)\overset{\delta^{\lle 1}}{\to} \br\le S,
\]
where the first (respectively, second) map is the connecting homomorphism induced by the sequence \eqref{res4} (respectively, \eqref{res3}). The lemma then follows from the commutativity of the diagram
\[
\xymatrix{\picxst(\lbe S\le)\ar[d]^{\delta^{\lle 0}}\ar[r]^(.45){\partial^{\le\lle 0}}&H^{\le 1}_{\et}\lbe(S,\mu^{\widetilde{Y}\!/\le Y})\ar[r]^(.63){\partial^{\lle 1}}&\br\le Y\ar@{=}[d]\\
H^{\le 1}_{\et}\lbe(S,\img \e\widetilde{q}\,)\ar[r]^(.6){\delta^{\lle 1}}&\br\le S \ar[r]^{(\le f\circ\e \iota\le)^{*}}&\br\le Y	
}
\]
which, in turn, follows from the functoriality of \'etale cohomology.	
\end{proof}

The sequence \eqref{seq3} induces the following exact sequences of abelian groups:
\begin{equation}\label{seq4}
0\to \pic \widetilde{X}/\e\pic S\to \picxst(\lbe S\le)\overset{\!\overline{d}_{\lle 2,\lle \widetilde{X}}^{\, 0,\le 1}}{\to}\bxsp\to 0,
\end{equation}
\begin{equation}\label{seq5}
0\to\br S/\e\bxsp\to\brop\to\brap\to 0
\end{equation}
and 
\begin{equation}\label{seq6}
0\to \brap\to H^{\le 1}_{\et}\lbe(S,\picxst)\to H^{\le\lle 3}\lbe(\lbe X\be/\be S\e),
\end{equation}
where $\overline{d}_{\lle 2,\lle \widetilde{X}}^{\, 0,\le 1}$ is the left-hand vertical map in diagram \eqref{cod} and $\brap$ is the group \eqref{bra}. The sequences \eqref{seq4}-\eqref{seq6} and their analogs for $f\colon X\to S$ yield the following exact and commutative diagrams of abelian groups:
\begin{equation}\label{ali2}
\xymatrix{0\ar[r]&\pic X/\le\pic S\ar@{->>}[d]\ar[r]&\picxs(\lbe S\le)\ar[d]\ar[r]&\bxs\ar@{^{(}->}[d]^{\subseteq}\ar[r]&0\\
0\ar[r]&\pic\widetilde{X}/\le\pic S\ar[r]&\picxst(\lbe S\le)\ar[r]^{\overline{d}_{\lle 2,\lle \widetilde{X}}^{\, 0,\le 1}}&\bxsp\ar[r]&0,
}
\end{equation}
\begin{equation}\label{ali}
\xymatrix{0\ar[r]&\br\le S/\e\bxs\ar@{->>}[d]\ar[r]&\bro\ar[d]^{\varphi_{1}^{*}}\ar[r]&\bra\ar[d]^{\varphi_{\rm a}^{*}}\ar[r]&0\\
0\ar[r]&\br\le S/\e\bxsp\ar[r]&\brop\ar[r]&\brap\ar[r]&0,
}
\end{equation}
and
\begin{equation}\label{dia}
\xymatrix{0\ar[r]&\bra\ar[d]^{\varphi_{\lbe\rm a}^{*}}\ar[r]&H^{\le 1}_{\et}\lbe(S,\picxs)\ar[d]^{\varphi^{\lle *}}\ar[r]&H^{\le 3}_{\et}\lbe(\le S,\G_{\le m})\ar@{=}[d]\\
0\ar[r]&\brap\ar[r]&H^{\le 1}_{\et}\lbe(S,\picxst)\ar[r]&H^{\le 3}_{\et}\lbe(\le S,\G_{\le m}).
}
\end{equation}
The first vertical map in diagram \eqref{ali2} is surjective by \cite[${\rm IV}_{\be 4}$, Proposition 21.8.5(ii)]{ega}, whose hypotheses are satisfied since
$X\setminus Y$ is schematically dense in $X$ by \cite[Proposition 5.4.3, p.~284]{ega1}. Consequently, the map $\overline{d}_{\lle 2,\lle \widetilde{X}}^{\, 0,\le 1}$ in diagram \eqref{ali2} induces an isomorphism of abelian groups
\begin{equation}\label{ppr}
\check{d}_{\lle 2,\lle \widetilde{X}}^{\,\e 0,\le 1}\colon \cok\be[\e\picxs(\lbe S\le)\to\picxst(\lbe S\le)]\isoto\bxsp/\e\bxs.
\end{equation}
Let
\begin{equation}\label{comp}
\alpha\colon \bxsp/\e\bxs\into H^{\le 1}_{\et}\lbe(S,\mu^{\widetilde{Y}\!/\le Y})
\end{equation}
be the composition
\[
\bxsp/\e\bxs\underset{\!\!\sim}{\e\overset{\!\big(\!\check{d}_{ 2,\lle \widetilde{X}}^{\,\le 0,\le 1}\be\big)^{\! -1}}{\lra}}\cok\be[\e\picxs(\lbe S\le)\to\picxst(\lbe S\le)]\overset{\e\partial^{\e 0}}{\into}H^{\le 1}_{\et}\lbe(S,\mu^{\widetilde{Y}\!/\le Y}),
\]
where $\check{d}_{\lle 2,\lle \widetilde{X}}^{\, 0,\le 1}$ is the isomorphism \eqref{ppr} and $\partial^{\e 0}$ is the second map in \eqref{seq2}. 

\begin{theorem}\label{choo} There exists a canonical exact sequence of abelian groups
\[
0\to\cok\!\!\be\left[\frac{\bxsp}{\bxs}\overset{\alpha}{\hookrightarrow} H^{\le 1}_{\et}\lbe(S,\mu^{\widetilde{Y}\!/\le Y})\right]\to\bra
\overset{\!\!\!\varphi_{\lbe\rm a}^{*}}{\to}\brap\to H^{\le 2}_{\et}\lbe(S,\mu^{\widetilde{Y}\!/\le Y}),
\]
where $\alpha$ is the map \eqref{comp}.
\end{theorem}
\begin{proof} The sequence \eqref{seq2} induces an exact sequence of abelian groups
\[
\frac{\bxsp}{\bxs}\overset{\alpha}{\hookrightarrow} H^{\le 1}_{\et}\lbe(S,\mu^{\widetilde{Y}\!/\le Y})\to  H^{\le 1}_{\et}\lbe(S,\picxs)\overset{\!\varphi^{\lle *}}{\to} H^{\le 1}_{\et}\lbe(S,\picxst)\to H^{\le 2}_{\et}\lbe(S,\mu^{\widetilde{Y}\!/\le Y}),
\]
where $\alpha$ is the composition \eqref{comp}. On the other hand, diagram \eqref{dia} yields a canonical isomorphism $\krn\varphi_{\lbe\rm a}^{*}\isoto\krn \varphi^{\lle *}$ and an injection $\cok \varphi_{\lbe\rm a}^{*}\hookrightarrow\cok\varphi^{\lle *}$. Thus there exist a canonical isomorphism
\[
\krn\varphi_{\lbe\rm a}^{*}\simeq\cok\!\!\be\left[\frac{\bxsp}{\bxs}\overset{\alpha}{\hookrightarrow} H^{\le 1}_{\et}\lbe(S,\mu^{\widetilde{Y}\!/\le Y})\right]
\]
and a canonical injection $\cok \varphi_{\lbe\rm a}^{*}\hookrightarrow H^{\le 2}_{\et}\lbe(S,\mu^{\widetilde{Y}\!/\le Y})$. The theorem is now clear.
\end{proof}

\begin{corollary}\label{fund} There exists a canonical exact sequence of abelian groups
\[
0\to \krn \varphi_{\lbe 1}^{*}\to \bro\!\overset{\!\!\varphi_{1}^{*}}{\lra}\!\brop\!\to\! \cok \varphi_{\lbe\rm a}^{*}\to 0,
\]
where $\krn \varphi_{1}^{*}$ is an extension
\[
0\to\frac{\bxsp}{\bxs}\to \krn \varphi_{1}^{*}\to \cok\!\!\be\left[\frac{\bxsp}{\bxs}\overset{\alpha}{\hookrightarrow} H^{\le 1}_{\et}\lbe(S,\mu^{\widetilde{Y}\!/\le Y})\right]\to 0.
\]
\end{corollary}
\begin{proof} An application of the snake lemma to diagram \eqref{ali}
yields an exact sequence of abelian groups
\[
0\to\frac{\bxsp}{\bxs}\to \krn \varphi_{1}^{*}\to \krn \varphi_{\lbe\rm a}^{*}\to 0
\]
and an isomorphism $\cok \varphi_{1}^{*}=\cok \varphi_{\lbe\rm a}^{*}$. The corollary is now immediate.
\end{proof}

The above arguments also yield the following generalization of \cite[Theorem 1]{bk} (see Remark \ref{bkr} below).

\begin{theorem}\label{gbk} If $\pic\e \widetilde{Y}=\psi^{\lle *}\be(\lle\pic\e Y\lle)$, then
\[
\bxs=\bxsp\cap\bys,
\]
where the intersection takes place in $\br\le S$.
\end{theorem}
\begin{proof} Diagram \eqref{zwe} shows that $\bxs\subseteq\bxsp\cap\bys$. Now, by Lemma \ref{comm}\,, the following diagram commutes (up to sign)
\begin{equation}\label{kbg}
\xymatrix{\bxsp/\e\bxs\,\ar[d]^(.45){\beta}\ar@{^{(}->}[r]^(.55){\alpha}&H^{\le 1}_{\et}\lbe(S,\mu^{\widetilde{Y}\!/\le Y})\ar[d]^(.45){\overline{\partial}^{\le 1}}_{\simeq}\\
\bysp/\e\bys\,\ar@{^{(}->}[r]^(.6){\overline{(\le f\lle\circ\le i\e)^{*}}}&\br\lbe(\e\widetilde{Y}\!/Y\le),
}
\end{equation}
where $\beta$ is induced by the inclusion $\bxsp\subseteq\bysp$, $\alpha$ is the map \eqref{comp}, $\overline{\partial}^{\e 1}$ is induced by the map $\partial^{\e 1}$ in the sequence \eqref{seq1} and $\overline{(\le f\lle\circ\le i\e)^{*}}$ is induced by the map $(\le f\lle\circ\le i\e)^{*}\colon \bysp\to \br\lbe(\e\widetilde{Y}\!/Y\le)$ in diagram \eqref{zwe}. The map $\overline{\partial}^{\e 1}$ above is an isomorphism by the exactness of \eqref{seq1} since $\pic\e \widetilde{Y}/\psi^{\lle *}\lbe(\lle\pic\e Y\lle)=0$. It follows that the map $\beta$ in \eqref{kbg} is injective, i.e., $\bxsp\cap\bys\subseteq\bxs$.
\end{proof}

\section{Pinched varieties}\label{abg}

Henceforth, we assume that $S=\spec k$, where $k$ is a field. We will write $p$ for the  characteristic exponent of $k$, i.e., $p=1$ if $\car\e k=0$ and $p=\car\e k$ otherwise. In the pinching diagram \eqref{pin}, $X$ and $\widetilde{X}$ are proper and geometrically integral FA $k$-schemes and $\psi\colon \widetilde{Y} \to Y$ is a morphism between finite $k$-schemes such that the induced morphism of Zariski sheaves $\cO_{\le Y}\to\psi_{*}\cO_{\le\widetilde{Y} }$ is injective. Then $Y=\spec A$ and $\widetilde{Y}=\spec \widetilde{A}\e$, where $A=\varGamma(\e Y,\cO_{\le Y}\lbe)\into\widetilde{A}=\varGamma(\e\widetilde{Y},\cO_{\le\widetilde{Y}}\lbe)$ are artinian $k\e$-\e algebras \cite[Proposition 6.5.4, p.~309]{ega1}. We can write $A=\prod_{\e y\e\in\e Y}\cO_{\e Y,\e y}$ and
$\widetilde{A}=\prod_{\e y\e\in\e Y}\prod_{\e\widetilde{y}\e\in\e \widetilde{Y}_{y}}\cO_{\e\widetilde{Y}\be,\,\widetilde{y}}\,$, where $\cO_{\e Y,\e y}$ and $\cO_{\e\widetilde{Y}\be,\,\widetilde{y}}$ are local artinian $k$-algebras \cite[(12.5), p.~331]{gw}. Since a local artinian ring is henselian (see, e.g., \cite[Corollary 9.1, p.~227]{eis} and \cite[04GE, Lemma 10.153.10]{sp}), the canonical map $H^{\le r}_{\et}\lbe(\cO_{\e Y,\e y},\G_{\le m})\to
H^{\le r}_{\et}\lbe(\kappa(y),\G_{\le m})$ is an isomorphism of abelian groups for every $r\geq 1$ and every $y\e\in\e Y$, and similarly for $\cO_{\e\widetilde{Y}\be,\,\widetilde{y}}$ \cite[Theorem 11.7(2), p.~181]{gb}. Thus the horizontal map in the following canonical commutative diagram of abelian groups is an isomorphism:
\begin{equation}\label{brd}
\xymatrix{H^{\le r}_{\le\lle\et}\lbe(\e k,\G_{\le m})\ar[dr]\ar[d]& \\
H^{\le r}_{\le\lle\et}\lbe(\e Y,\G_{\le m})	\ar[r]^(.4){\sim}& \displaystyle\prod_{\e y\e\in\e Y}H^{\le r}_{\le\lle\et}\be(\kappa(y),\G_{\le m}\lbe).
}
\end{equation}
For each $y\e\in\e Y$, the $y\e$-th component of the oblique map above is the pullback homomorphism $H^{\le r}_{\le\et}\lbe(k,\G_{\le m}\lbe)\to H^{\le r}_{\le\et}\lbe(\kappa(y),\G_{\le m})$ induced by the inclusion $k\subseteq \kappa(y)$. 
Setting $r=1$ in \eqref{brd} (and in its analog for $\widetilde{Y}\e$), we obtain $\pic\e Y=\pic\e \widetilde{Y}=0$, whence
\begin{equation}\label{pac}
H^{\le 1}\lbe(k_{\le\et},\mu^{\widetilde{Y}\!/\le Y})\isoto \br(\e\widetilde{Y}\!/\e Y\e)
\end{equation}
by the exactness of the sequence \eqref{seq1}. Further, the vertical and oblique maps in \eqref{brd} have the same kernel. In particular, setting $r=2$ in \eqref{brd}, we obtain the equality
\begin{equation}\label{bk0}
\br(\le Y\!/\e k\le)=\bigcap_{\e y\e\in\e Y}\br\lbe(\kappa(\lle y)\lbe/k\lle),
\end{equation}
where the intersection takes place in $\br\le k$.

\begin{theorem}\label{gbk2} The following equality between subgroups of $\br\le k$ holds:
\[
\bxk=\bxkp \cap \bigcap_{\e y\e\in\e Y}\br\lbe(\kappa(y)\lbe/k\lle).
\]	
\end{theorem}
\begin{proof} Since $\pic\e \widetilde{Y}=\psi^{\lle *}\lbe(\le\pic\e Y\lle)=0$, the theorem follows at once from \eqref{bk0} and Theorem \ref{gbk}\,.
\end{proof}

\begin{remark}\label{bkr} The above statement generalizes \cite[Theorem 1]{bk}, where the equality of Theorem \ref{gbk2} is obtained under the following assumptions: $\car k=0$, $X$ is a $k$-curve with normalization $\widetilde{X}$ and $Y$ is the non-normal locus of $X$ (note that in \cite[Theorem 1]{bk} the groups $(\pic\le\widetilde{X}^{\rm s}\le)^{\lbe\g}\!/\e\pic\le \widetilde{X}$ and $\bxkp$ have been identified via \eqref{cbx}, and similarly for $X$).
\end{remark}

We now observe that, for every $r\geq 1$, there exists a canonical commutative diagram of abelian groups
\begin{equation}\label{bdir}
\xymatrix{H^{\le r}_{\et}\lbe(\e Y,\G_{\le m})\ar[r]^(.4){\sim}\ar[d]^(.45){\psi^{\lle*}}& \displaystyle\prod_{\e y\e\in\e Y}H^{\le r}_{\et}\lbe(\kappa(y),\G_{\le m})\ar[d]^{\prod_{\e y\lle\in\lle Y}\lbe \rho_{y}^{\lle *}}\\
H^{\le r}_{\et}\lbe(\e\widetilde{Y},\G_{\le m})	\ar[r]^(.33){\sim}& \displaystyle\prod_{\e y\e\in\e Y}\prod_{\e\widetilde{y}\e\in\e \widetilde{Y}_{y}}H^{\le r}_{\et}\lbe(\kappa(\le\widetilde{y}\,),\G_{\le m})
}
\end{equation}
where, for each $y\e\in\e Y$,
$\rho_{y}^{*}\colon H^{\le r}_{\et}\lbe(\kappa(y),\G_{\le m})\to\prod_{\e\widetilde{y}\e\in\e \widetilde{Y}_{y}}\lbe H^{\le r}_{\et}\lbe(\kappa(\le\widetilde{y}\,),\G_{\le m})$
is the map whose $\widetilde{y}\le$-\le component, for each $\widetilde{y}\e\in\e \widetilde{Y}_{y}$, is the pullback homomorphism $H^{\e r}_{\et}\be(\kappa(y),\G_{\le m}\lbe)\to H^{\le r}_{\et}\lbe(\kappa(\le\widetilde{y}\,),\G_{\le m})$ induced by the inclusion $\kappa(y)\subseteq \kappa(\le\widetilde{y}\,)$. Clearly,
\eqref{bdir} induces isomorphisms of abelian groups
\begin{equation}\label{relr}
H^{\le r}\be(\e \widetilde{Y}\!/Y\le)\isoto\prod_{y\e\in\e Y}\bigcap_{\e\widetilde{y}\e\in\e \widetilde{Y}_{y}}\!H^{\e r}\be(\kappa(\le\widetilde{y}\,)/\lbe \kappa(y))
\end{equation}
(where, for each $y\in Y$, the intersection takes place in $H^{\e r}_{\et}\be(\kappa(y),\G_{\le m})$) and
\begin{equation}\label{corelr}
\frac{H^{\le r}_{\et}\lbe(\e\widetilde{Y},\G_{\le m})}{\psi^{\le*}\lbe(H^{\le r}_{\et}\lbe(\e Y,\G_{\le m}))}\isoto\prod_{y\e\in\e Y}\cok\!\!\lbe\left[\e H^{\e r}_{\et}\lbe(\kappa(y),\G_{\le m})\to \textstyle{\prod_{\,\widetilde{y}\e\in\e \widetilde{Y}_{y}}}\e H^{\e r}_{\et}\lbe(\kappa(\le\widetilde{y}\,),\G_{\le m})\right]\!.
\end{equation}
Setting $r=2$ in \eqref{relr} and \eqref{corelr}, we obtain canonical isomorphisms
\begin{equation}\label{rel}
\hskip -1.2cm\br(\e\widetilde{Y}\!/\e Y\le)\isoto\prod_{y\e\in\e Y}\bigcap_{\e\widetilde{y}\e\in\e \widetilde{Y}_{y}}\!\!\br\lbe(\kappa(\widetilde{y}\,)/\kappa(y))
\end{equation}
and
\begin{equation}\label{corel}
\hskip 1.25cm\frac{\br\e\widetilde{Y}}{\psi^{\le\lle*}\lbe(\br\e Y\lle)}\isoto\prod_{y\e\in\e Y}\!\cok\!\!\lbe\left[\e\br\e \kappa(y)\to\textstyle{\prod_{\e\widetilde{y}\e\in\e \widetilde{Y}_{y}}}\,\br\e \kappa(\le\widetilde{y}\,)\right]\!.
\end{equation}

\begin{theorem}\label{fund2} The pinching diagram \eqref{pin} induces an exact sequence of abelian groups
\[
\hskip -.5cm 0\!\to\!\cok\!\!\!\left[\le \frac{\bxkp}{\bxk}\into \prod_{y\e\in\e Y}\!\lbe\bigcap_{\e\widetilde{y}\e\in\e \widetilde{Y}_{y}}\!\!\br\lbe(\kappa(\widetilde{y}\,)/\kappa(y))\lbe\right]\!\!\to\!\bra
\overset{\!\!\!\varphi_{\lbe\rm a}^{*}}{\to}\!\brap\!\to H^{\le 2}_{\et}\lbe(k,\mu^{\widetilde{Y}\!/\le Y}\lbe),
\]
where $H^{\le 2}_{\et}\lbe(k,\mu^{\widetilde{Y}\!/\le Y})$ is an extension
\[
\hskip 0.1cm 0\to\!\prod_{\e y\e\in\e Y}\!\cok\!\!\!\left[\be\br\le \kappa(y)\!\to\!\!\prod_{\e\widetilde{y}\e\in\e \widetilde{Y}_{y}}\!\!\br\e \kappa(\le\widetilde{y}\,)\right]\!\to\! H^{\le 2}_{\et}\lbe(k,\mu^{\widetilde{Y}\!/\le Y})\!\to\!\prod_{y\e\in\e Y}\be\bigcap_{\e\widetilde{y}\e\in\e \widetilde{Y}_{y}}\!\!H^{\le 3}\lbe(\kappa(\le\widetilde{y}\,)/\lbe \kappa(y))\to 0.
\]
\end{theorem}
\begin{proof} By \eqref{pac} and \eqref{rel}, there exist canonical isomorphisms
\begin{equation}\label{cap}
H^{\le 1}_{\le\et}\lbe(k,\mu^{\widetilde{Y}\!/\le Y})\isoto \br(\e\widetilde{Y}\!/\e Y\le)\isoto \prod_{y\e\in\e Y}\bigcap_{\e\widetilde{y}\e\in\e \widetilde{Y}_{y}}\!\br\lbe(\kappa(\widetilde{y}\,)/\kappa(y)).
\end{equation}
On the other hand, the exactness of \eqref{seq1} yields an extension
\[
0\to \frac{\br\e\widetilde{Y}}{\psi^{\lle*}\lbe(\br\e Y\le)}\to H^{\le 2}_{\le\et}\lbe(k,\mu^{\widetilde{Y}\!/\le Y})\to H^{\le 3}\lbe(\e \widetilde{Y}\!/Y\lle)\to 0.
\]
The theorem is now immediate from Theorem \ref{fund} via the isomorphisms \eqref{relr} (for $r=3$) and \eqref{corel}.
\end{proof}

\smallskip

\begin{corollary}\label{fund3} There exists a canonical exact sequence of abelian groups
\[
0\to \krn \varphi_{1}^{*}\to \brok\!\overset{\!\!\varphi_{1}^{*}}{\lra}\!\brokp\!\to\! \cok \varphi_{\lbe\rm a}^{*}\to 0,
\]
where $\krn \varphi_{1}^{*}$ is an extension
\[
0\to\frac{\bxkp}{\bxk}\to \krn \varphi_{1}^{*}\to \cok\!\!\!\left[\le \frac{\bxkp}{\bxk}\into \prod_{y\e\in\e Y}\!\lbe\bigcap_{\e\widetilde{y}\e\in\e \widetilde{Y}_{y}}\!\!\br\lbe(\kappa(\widetilde{y}\,)/\kappa(y))\lbe\right]\to 0.
\]
\end{corollary}
\begin{proof} The corollary follows from Theorem \ref{fund2} in the same way that Corollary \ref{fund} follows from Theorem \ref{choo}.
\end{proof}

\smallskip

\begin{remark}\label{alc} For each $y\in Y$, we have $I\lbe\big(\le\widetilde{Y}_{\lbe y}\lbe\lle\big)=\gcd\{\e [\le\kappa\le(\le\widetilde{y}\,)\be\colon\! \kappa(y)\le]\colon \widetilde{y}\e\in\e \widetilde{Y}_{y}\e\}$, whence $I\lbe\big(\le\widetilde{Y}_{\lbe y}\lbe\lle\big)\cdot\left(\le\bigcap_{\,\e\widetilde{y}\e\in\e \widetilde{Y}_{\lbe y}}\! H^{\le r}\lbe(\kappa(\le\widetilde{y}\,)/\lbe \kappa(y))\!\right)\!=0$ for all $r\geq 1$ by \eqref{ihrk}. Consequently, if $m(\e \widetilde{Y}\!/Y\le)\defeq\lcm\{\le I\lbe\big(\le\widetilde{Y}_{\lbe y}\lbe\lle\big)\colon y\in Y\lle\}$, then $\prod_{\, y\e\in\e Y}\be\bigcap_{\,\e\widetilde{y}\e\in\e \widetilde{Y}_{y}}\be H^{\le r}\be(\kappa(\le\widetilde{y}\,)/\lbe \kappa(y))$ is an $m(\e \widetilde{Y}\!/Y\le)\e$-\e torsion group for every $r\geq 1$. In particular, $\prod_{\,y\e\in\e Y}\bigcap_{\,\widetilde{y}\e\in\e \widetilde{Y}_{y}}\br\lbe(\kappa(\widetilde{y}\,)/\kappa(y))$, and therefore also $\bxkp/\le \bxk$ and $H^{\le 1}_{\le\et}\lbe(k,\mu^{\widetilde{Y}\!/\le Y})$ \eqref{cap}, are $m(\e \widetilde{Y}\!/Y\le)\e$-\e torsion groups.
\end{remark}

\smallskip

We will now derive some concrete consequences of the preceding general considerations.

\smallskip

If $\psi\colon \widetilde{Y}\to Y$ is a {\it universal homeomorphism} then, for every $y\e\in\e Y$, $\psi^{\le-1}\lbe(y)$ is a single (closed) point and $\kappa(\psi^{\le-1}\lbe(y))/\kappa(y)$ is a finite and purely inseparable extension \cite[04DC, Lemma 29.45.5 and 01S2, Lemma 29.10.2]{sp}. Set 
$[\e \kappa(\psi^{\le-1}\lbe(y))\colon\! \kappa(y)\le]\!=\!p^{\le n_{\lle y}}$,  
where $n_{\lle y}\geq 0$ is an integer. Then, by \cite[Corollary 11.4.2, p.~437]{sa}, we have $\br\lbe(\kappa(\psi^{\le-1}\lbe(y))/\kappa(y))=\br\lbe(\kappa(y))_{p^{\lle n_{\lbe\le y}}}$, whence
\begin{equation}\label{upa}
\prod_{y\e\in\e Y}\!\lbe\bigcap_{\e\widetilde{y}\e\in\e \widetilde{Y}_{y}}\!\!\br\lbe(\kappa(\widetilde{y}\,)/\kappa(y))=\prod_{\e y\e\in\e Y}\be\br\lbe(\kappa(y))_{\lle p^{\lle n_{\lbe\lle y}}}.
\end{equation}

\begin{proposition}\label{uht} Assume that $\psi\colon \widetilde{Y}\to Y$ is a {\rm universal homeomorphism}. Then there exists a canonical exact sequence of abelian groups
\[
0\to \krn \varphi_{1}^{*}\to \brok\!\overset{\!\!\varphi_{1}^{*}}{\lra}\!\brokp\!\to\! 0.
\]
The group $\krn \varphi_{1}^{*}$ is an extension
\[
0\to\frac{\bxkp}{\bxk}\to \krn \varphi_{1}^{*}\to \cok\!\!\be\left[\le \frac{\bxkp}{\bxk}\into \prod_{\e y\e\in\e Y}\be\br\lbe(\kappa(y))_{\lle p^{\lle n_{\lbe\le\lle y}}}\!\right]\to 0
\]
where, for each $y\in Y$, $p^{\e n_{\le y}}=[\e \kappa(\psi^{\le-1}\lbe(y))\colon\! \kappa(y)\le]$.
\end{proposition}
\begin{proof} By \cite[Proposition 4.11 and Remark 4.12]{bri15}, $\mu^{\widetilde{Y}\!/\e Y}$ is a commutative, smooth and connected unipotent algebraic $k\e$-\e group. Consequently, $H^{\lle 2}\lbe(k_{\le\et},\mu^{\widetilde{Y}\!/\e Y})=0$ by Lemma \ref{spl}(ii). The proposition is now immediate from Theorem \ref{fund2} and Corollary \eqref{fund3} using \eqref{upa}\,.
\end{proof}

\begin{corollary}\label{uht2} Assume that $\psi\colon \widetilde{Y}\to Y$ is a universal homeomorphism and that at least one of the following conditions holds:
\begin{enumerate}
\item[(i)] $k$ is perfect, or
\item[(ii)] $\psi$ induces isomorphisms on residue fields.
\end{enumerate}
Then $\varphi_{1}^{*}\colon \brok\to\brokp$ is an isomorphism of abelian groups.
\end{corollary}
\begin{proof} If $k$ is perfect, then $\kappa(y)$ is perfect for every $y\e\in\e Y$ by \cite[Corollary 1, p.~A.V.43]{bca}. Consequently, $\br\e\kappa(y)$ contains no nontrivial $p\e$-torsion elements by \cite[Corollary 11.4.5, p.~439]{sa}, whence $\br\lbe(\kappa(y))_{p^{n_{\lbe\le y}}}=0$ for every $y\in Y$. In case (ii) we have $p^{\le n_{\lbe\le y}}=[\e \kappa(\psi^{\le-1}\lbe(y))\colon\! \kappa(y)\le]=1$ for all $y\in Y$, whence $\br\lbe(\kappa(y))_{p^{n_{\lbe\le y}}}=0$ as well for every $y\in Y$. The corollary is now immediate from the proposition.
\end{proof}

\smallskip

In order to state the final result of this section, we need the following definitions.

\smallskip

A finite morphism of $k$-schemes $X^{\prime}\to X$ is called a {\it partial seminormalization} (of $X$) if (a) $X^{\prime}$ is reduced, and (b) for every $x\in X$, $(X^{\prime}\times_{X}\spec \kappa(x))_{\rm red}=\spec \kappa( x^{\le\prime}\le)$ is a one\e-\le point scheme and the induced map $\kappa(x)\to\kappa( x^{\le\prime}\le)$ is an isomorphism \cite[Definition 10.11]{kk}. By \cite[p.~307]{kk}, there exists a unique largest partial seminormalization of $X$ that dominates every other partial seminormalization of $X$. It is denoted by $\sigma\colon X^{ S\lbe N}\to X$ and called the {\it seminormalization morphism of $X$}. If $Y$ is the non\e-\le normal locus of $X$, then $X^{\lbe S\lbe N}$ can be obtained from $X^{\lbe N}$ via a pinching diagram of the form
\begin{equation}\label{dsn}
\xymatrix{(\lbe X^{\lbe N}\!\be\times_{\lbe X}\be\! Y\e)_{\rm red}\,\ar@{^{(}->}[r]\ar[d]&
X^{\lle\be N}\ar[d]\\
Z\,\ar@{^{(}->}[r]&X^{\lbe S\lbe N}
}
\end{equation}
for a certain reduced $Y$-scheme $Z$. See \cite[\S10.18, p.~310]{kk} for more details.
In particular, seminormal curves are obtained from their normalizations by pinching {\it points} only (to obtain an arbitrary curve $X$ from $X^{\lbe N}\!$ via pinching, one also needs to pinch {\it tangent vectors} on $X^{\lbe N}$ to recover any ``cuspidal" singularities that $X$ may have. See \cite[\S1]{ss}).

\begin{proposition}\label{sn} Let $k$ be any field and let $X$ be a proper and geometrically integral {\rm FA} $k\le$-\le scheme with a finite non\e-\le normal locus. If $\sigma\colon X^{\lbe S\lbe N}\to X$ is the seminormalization morphism of $X$, then the pullback map $\sigma^{\le *}\colon \bro \to \br_{\lbe 1} X^{\lbe S\lbe N}$ is an isomorphism of abelian groups.
\end{proposition}
\begin{proof} By \cite[0BXQ, Lemma 33.27.1(1), and 035E, Lemma 29.54.5(2)]{sp}, the normalization morphism $\nu\colon X^{\lle\be N}\to X$ is finite and surjective. Thus, by \cite[II, Proposition 5.4.2(ii), Theorem 5.5.3(i) and Corollary 6.1.11]{ega} and \cite[\S2.2(3)]{gll}, $X^{\lle\be N}$ is a proper FA $k\e$-\le scheme. Further, by Lemma \ref{bri}\,, $X^{\lle\be N}$ is geometrically integral. Thus $X^{\lle\be N}$ is a proper and geometrically integral FA $k\e$-scheme. Consequently, \eqref{dsn} is a pinching diagram of the form \eqref{pin}, whence $X^{\lbe S\lbe N}$ is also a proper and geometrically integral FA $k\e$-scheme (see the beginning of section \ref{piag}). We now observe that Overkamp \cite[Theorem 2.24(i)]{ov21} has partially generalized \cite[Lemma 0C1L]{sp} by showing that, if $k$ is any field and $\dim X=1$, then the seminormalization morphism $\sigma\colon X^{\lbe S\lbe N}\to X$ factors as
\begin{equation}\label{of}
X^{\lbe S\lbe N}=X_{1}\overset{\!\!\sigma_{\lbe\lle 1}}{\lra} X_{2}\overset{\!\!\sigma_{ 2}}{\lra}\dots \overset{\!\!\sigma_{\lbe\lle n}}{\lra} X_{n+1}=X
\end{equation}
for some $n\in\N$ where, for each $i=1,\dots,n$, $X_{i+1}$ is obtained from $X_{i}$ via a pinching diagram of the form
\begin{equation}\label{dsn2}
\xymatrix{\spec\!\!\lbe\left (\!\frac{\kappa(x_{i+1})[\varepsilon]}{(\varepsilon^{2})}\!\right)\ar@{^{(}->}[r]\ar[d]_{\psi_{i}}& X_{i}	\ar[d]_{\sigma_{\lbe\lle i}}\\
\spec \kappa(x_{i+1})\,\ar@{^{(}->}[r]& X_{i+1}
}
\end{equation}
for some closed point $x_{i+1}\in X_{i+1}$. A close examination of the proofs of \cite[Proposition 2.22, Corollary 2.23 and Theorem 2.24(i)]{ov21} reveals that these proofs (and therefore the factorization \eqref{of} and the associated diagrams \eqref{dsn2} above) remain valid if the hypothesis $\dim X=1$ is replaced with the condition that the cokernel of the canonical morphism $\cO_{\lbe X}\into\nu_{*}\cO_{\be X^{\lle\be N}}$ be of finite length, which certainly holds since $X$ has a finite non\le-\le normal locus.
Next we observe that, since $X_{1}=X^{\lbe S\lbe N}$ is a proper and geometrically integral FA $k\e$-scheme, diagram \eqref{dsn2} with $i=1$ is a pinching diagram of the form \eqref{pin}. It then follows that each of the diagrams \eqref{dsn2} is of the form \eqref{pin}. Finally, since the map $\psi_{i}$ in \eqref{dsn2} is a universal homeomorphism that induces isomorphisms on residue fields, Corollary \ref{uht2}(ii) shows that the pullback map $\sigma_{\lbe\lle i}^{*}\colon \br_{\lbe 1} X_{i+1}\to \br_{\lbe 1} X_{i}$ is an isomorphism for every $i=1,\dots,n$. Since $\sigma^{*}=\sigma_{\lbe\lle 1}^{*}\circ\dots\circ\sigma_{\lbe\lle n}^{*}$, the proposition follows.
\end{proof}

\section{Examples}

\begin{example}\label{ex1}  In the pinching diagram \eqref{pin} with $S=\spec k$, assume that  $\widetilde{X}$ is smooth and universally ${\rm CH}_{\le 0}\e$-trivial, i.e., the degree map ${\rm CH}_{\le 0}\lle(\lbe\widetilde{X}_{F}\be)\to\Z$ is an isomorphism for every field extension $F/k$. Examples include retract $k$-rational varieties, i.e., those $k$-schemes $\widetilde{X}$ for which there exists a dominant rational map $\Pn_{\! k}^{\e n}\dashrightarrow \widetilde{X}$ with a right inverse $\widetilde{X}\dashrightarrow \Pn_{\! k}^{\e n}$ for some $n\geq 1$. See \cite[\S1.2, pp.~34-35]{actp}\,. By \cite[Theorem 1.1]{a3b}, the pullback map $\br\le F\to \br\widetilde{X}_{\lbe F}$ is an isomorphism for every field extension $F/k$, whence $\bxkp=\brap=0$. Thus Corollary \ref{fund3} yields a canonical exact sequence of abelian groups
\[
0\to \prod_{\e y\e\in\e Y}\bigcap_{\,\e \widetilde{y}\e\in\e \widetilde{Y}_{y}}\!\br\lbe(\kappa(\widetilde{y}\,)/\kappa(\le y))\to \brok\to\br\e k\to 0.
\]	
The third map in the above sequence admits a canonical homomorphic section, namely the pullback map $f^{\lle*}\colon\br\e k\to \brok$. Thus there exists a canonical isomorphism of abelian groups
\[
\brok=\br\e k\oplus \prod_{y\e\in\e Y}\bigcap_{\e \widetilde{y}\e\in\e \widetilde{Y}_{y}}\!\br\lbe(\kappa(\widetilde{y}\,)/\kappa(\le y)).
\] 
\end{example}

\begin{example}\label{ex2} In the pinching diagram \eqref{pin} with $S=\spec k$, assume that  $\widetilde{X}$ is a Severi-Brauer $k$-variety, i.e., $\widetilde{X}^{\rm s}\simeq\Pn_{\! k^{\le\rm s}}^{\e\le{\rm dim}\le\lle \widetilde{X}}$. Then the Amitsur subgroup of $\widetilde{X}$ in $\br\e k$ is the finite cyclic subgroup of $\br\e k$ generated by the class of $\widetilde{X}$, i.e., $\bxkp=\langle\widetilde{X}\e\rangle$. See \cite[Theorem 5.4.1]{gs} and \cite[Theorem 3.5.5]{ctsk}. Further, the pullback map $\widetilde{f}^{\le *}\colon\br\e k\to\brop$ is surjective, i.e., $\brap=0$ \cite[p.~183]{ctsk}, whence $\widetilde{f}^{\le *}$ induces an isomorphism of abelian groups $\widetilde{g}\colon\br\e k/\langle\widetilde{X}\le\rangle\isoto\brop$. Let $g=
\widetilde{g}^{\e-1}\circ \varphi_{1}^{*}\colon \brok\to \br\e k/\langle\widetilde{X}\rangle$, where $X$ is as in diagram \eqref{pin}. Then, by Corollary \ref{fund3}\,, the algebraic Brauer group of the pinched Severi-Brauer variety $X$ can be described as follows: there exists a canonical exact sequence of abelian groups
\[
0\to \krn g\to\brok\overset{\!g}{\to}\br\e k/\langle\widetilde{X}\le\rangle\to 0,
\]
where $\krn g$ is an extension
\[
\hskip 1.5cm 0\to\frac{\langle\widetilde{X}\rangle}{\bxk}\to \krn g\to \cok\!\!\!\left[\le \frac{\langle\widetilde{X}\rangle}{\bxk}\into \prod_{y\e\in\e Y}\!\lbe\bigcap_{\e\widetilde{y}\e\in\e \widetilde{Y}_{y}}\!\!\br\lbe(\kappa(\widetilde{y}\,)/\kappa(y))\lbe\right]\to 0.
\]
\end{example}
\smallskip

\begin{example}\label{ex3} (cf. \cite[Example 3.23]{lau}) Let $k=\mathbb{F}_{\be 2}(t_{1}, t_{2})$ and let $G$ be the nontrivial form of $\G_{a,\le k}$ given by the equation $y^{\e 4}=x + t_{1}x^{\le 2}+ t_{2}^{2}\e x^{\le 4}$. The $k$-group $G$ is a $1\e$-\le dimensional $k$-wound group \cite[Definition B.2.1, p.~568]{cgp}. Let $\widetilde{X}$ be the normalization of the schematic closure of $G$ in $\Pn^{\le 2}_{\!k}$. Clearly, $\widetilde{X}(k)\neq\emptyset$, whence $\bxkp=0$. Now, since
$(\be\widetilde{X}_{\kbar})^{N}\simeq\Pn_{\be\kbar}^{\le 1}$ by \cite[proof of Lemma 1.1]{rus} and $\widetilde{X}_{\kbar}$ is obtained by pinching $(\be\widetilde{X}_{\kbar})^{N}$,  $\widetilde{X}_{\kbar}$ is integral, i.e., $\widetilde{X}$ is geometrically integral. Further, $\widetilde{X}-G=P_{\be\infty}$ is a non\e-\le smooth closed point with residue field $\kappa(P_{\be\infty})= \mathbb{F}_{\be 2}\big(t_{1}^{1\lbe/2}, t_{2}^{1\lbe/2}\le\lle\big)$ \cite[Example 3.23]{lau}. For every $c\in k$, set $K_{\lbe c}=k\big(t_{1}^{1\lbe/2}\be+c\e t_{2}^{1/2}\le\lle\big)\subset \kappa(P_{\be\infty})$ and let $X_{\lbe\lle c}$ be obtained by pinching $\widetilde{X}$ along $\spec \kappa(P_{\be\infty})$ via the canonical morphism $\spec \kappa(P_{\be\infty})\to\spec K_{c}\e$, i.e., $X_{c}$ is given by the pushout diagram
\[
\xymatrix{\spec \kappa(P_{\be\infty}) \ar@{^{(}->}[r]\ar[d]&\widetilde{X}\ar[d]^(.45){\nu}\\
\spec K_{\lbe c}\,\ar@{^{(}->}[r]&X_{\lbe\lle c}.
}
\]
The preceding diagram is of the form \eqref{pin} since the FA $k$-scheme $\widetilde{X}$ is proper and geometrically integral. Now, by \cite[Proposition 3.2]{lau}, the map $\nu$ above is the normalization morphism of $X_{c}\e$. Further, by \cite[Lemmas 2.3(iv) and 3.17 and Theorem 3.21(a)]{lau}, $X_{c}$ is a seminormal $k\le$-\le curve which is equipped with an almost homogeneous $G$-action, i.e., $X_{c}$ contains a homogeneous $G\le$-\le stable open subscheme. Now we observe that, since $\kappa(P_{\be\infty})=K_{c}\lbe\big(t_{2}^{1/2}\le\big)$, $\kappa(P_{\be\infty})/K_{c}$ is a purely inseparable quadratic extension. Consequently, Proposition \ref{uht} yields a canonical exact sequence of abelian groups
\[
0\to \br\lbe(\lbe K_{\lbe c})_{\lle 2}\to \br X_{c} \!\overset{\!\!\nu^{*}}{\lra}\!\br \widetilde{X}\!\to\! 0.
\]
Note that $\widetilde{X}$ is a regular non\le -\le smooth $k\e$-\le curve. For reasons that are explained in the next remark, and further illustrated in Example \ref{las} below, the methods of this paper are {\it useless} for determining the structure of $\br \widetilde{X}$.
\end{example}

\begin{remark}\label{rem} Let $k$ be any imperfect field and let $X$ be a proper and 
geometrically integral $k$-curve with non\e-\le normal locus $Y$. If $X^{\be N}$ is {\it geometrically normal}, i.e., smooth \cite[II, Corollary 7.4.5, and
${\rm IV}_{\be 4}$, Corollary 17.5.2]{ega}, Corollary \ref{fund3} applied to the normalization morphism $\nu_{X}\colon X^{\be N}\to X$ relates the Brauer group of $X$ to the (more accessible) Brauer group of the smooth curve $X^{\be N}$. However, there exist many proper and geometrically integral $k$-curves whose normalization is not smooth. See Example \ref{ex3} above and also \cite[Examples 6.8.2 and 6.12.3(2)]{kmt}, \cite[Example 1.22]{ach} and \cite[Example 3.1]{to}. In general, there exists a finite and purely inseparable extension $K/k$ such that the $K$-scheme $(\be X_{\lbe K}\be)^{ N}$ is smooth \cite[0BXQ, Lemma 33.27.4]{sp}. Corollary \ref{fund3} can then be applied to the triple $((\be X_{\lbe K}\be)^{ N}\!, X_{\lbe K}, \nu_{X_{K}}\!\times_{\lbe X_{\lbe K}}\be Y_{\lbe K})$ to relate the structure of $\br X_{\lbe K}$ to that of $\br(\be X_{\lbe K}\be)^{ N}$. However, the methods of this paper yield no information about $\br X$ itself, as we show in the next example.
\end{remark}

\begin{example}\label{las} Let $C$ be any nontrivial form of $\A^{\!1}_{\le k}$, i.e., $C_{k^{\le\lle\prime}}\simeq \A^{\!1}_{k^{\le\lle\prime}}$ for some nontrivial, finite and purely inseparable extension $k^{\e\prime}\!/k$. The indicated class of affine $k\le$-\le curves contains the underlying scheme of every $1\le$-\le dimensional $k\le$-\le wound group (e.g., the $k\le$-\le group $G$ discussed in Example \ref{ex3} above). See \cite[Proposition 4.7, p.~45]{kmt}. Now let $X$ be the regular completion of $C$ \cite[0BXX]{sp}. Then $P_{\be\infty}=X-C$ is a single (closed) point such that $k\subsetneq \kappa(\lbe P_{\be\infty}\lbe)\subseteq k^{\e\prime}$. See \cite[Lemma 1.1]{rus} and \cite[Remark 1.17]{ach}. The point $P_{\be\infty}$ is often non\e-\e smooth, in which case $X$ is a regular non\le-\le smooth $k\le$-\le curve. Next, let $K/k$ be any subextension of $k^{\e\prime}\!/k$ such that the regular completion of $C_{\lbe K}$, i.e., $(\be X_{\be K}\be)^{\lbe N}$ \cite[0BXX, Lemma 53.2.9]{sp}, is $\Pn_{\! K}^{\le 1}$ (e.g., $K=k^{\e\prime}\e$). Further, let $P$ (respectively, $Q$) be the unique point of $X_{\lbe K}$ (respectively, $\Pn^{\le 1}_{\!K}$) lying above $P_{\be\infty}$. Then $X_{\be K}$ fits into a pinching diagram
\[
\xymatrix{\Pn^{\le 1}_{\!K}\be\times_{\lbe X_{\lbe K}}\! Y_{\lbe K} \ar@{^{(}->}[r]\ar[d]_{\nu_{X_{\be K}}\!\times_{\lbe X_{\lbe K}}\be Y_{\lbe K}}&\Pn^{\le 1}_{\!K}	\ar[d]_{\nu_{\lle X_{\be K}}}\\
Y_{\be K}\,\ar@{^{(}->}[r]&X_{\lbe K},
}
\]
where $\nu_{X_{K}}\colon \Pn_{\! K}^{\le 1}\to X_{\lbe K}$ is the normalization morphism of $X_{K}$. The schemes $Y_{\lbe K}$ and $\Pn^{\le 1}_{\!K}\be\times_{\lbe X_{\lbe K}}\! Y_{\lbe K}$ are one\e-\le point schemes with underlying sets $P$ and $Q$, respectively \cite[\S4.1]{ach}. Since $K\subseteq \kappa(P)\subseteq \kappa(Q)\subseteq k^{\le\lle\prime}$, the extension $\kappa(Q)/\kappa(P)$ is purely inseparable, whence $\psi$ is a universal homeomorphism \cite[Proposition 4.11 and Remark 4.12]{bri15}. Thus, by \eqref{upa} and Corollary \ref{cor} (see also Example \ref{ex1}), there exists a canonical isomorphism of abelian groups
\[
\br X_{\lbe K}=\br K\be\oplus\be \br\lbe(\lbe\kappa(\be P))_{p^{\le d}},
\]
where $p^{\le d}\defeq[\le\kappa(Q)\colon\be \kappa(P)\le]$ divides $[\e k^{\e\lle\prime}\lbe\colon\!\be K\e]$. The preceding isomorphism essentially determines the structure of $\br X_{\lbe K}$. However, in order to determine the structure of $\br X$, one needs to understand the structure of $\br\lbe(\lbe X_{\be K}\lbe/\be X\lle)=\krn\![\e\br X\!\to\! \br X_{\lbe K}]$, which requires methods (to be discussed elsewhere) that are essentially different from those of this paper.
\end{example}

\section{The Roquette\e-Lichtenbaum theorem for a class of singular curves}\label{cin}

Unless stated otherwise, in this section $k$ denotes a non-archimedean local field, i.e., a finite extension of either $\Q_{\lle p}$ or $\mathbb{F}_{\be p}((t))$ for some prime number $p$. We exclude from our discussion the archimedean local fields since the results of this section over such fields are either trivial (when $k\simeq\mathbb C\le$) or already contained in \cite{bk} (when $k\simeq\mathbb R\le$).

\smallskip

The invariant map of local class field theory
\begin{equation}\label{inv}
{\rm inv}\colon \br\e k\isoto\Q/\le\Z
\end{equation}
is an isomorphism of abelian groups. In particular, every subgroup of finite exponent of $\br\e k$ is finite and cyclic, i.e., its order is equal to its exponent. If $X$ is an algebraic $k$-scheme, then $\bxk$ \eqref{relb} is a subgroup of $\br\e k$ that is annihilated by the index $I(\be X\lle)$ of $X$. Consequently, $\bxk$ is finite and cyclic and its order divides $I(\be X\lle)$.

If $K/k$ is a finite field extension, then \eqref{inv} induces an isomorphism of abelian groups
\begin{equation}\label{inv2}
\br\lbe(\be K\!/k)\isoto [K\colon\be k\e]^{-1}\Z/\le\Z\e.
\end{equation}
See \cite[Theorem 37, p.~155]{sh}.

\begin{theorem} \label{roq} {\rm (Roquette\e-Lichtenbaum)} Let $X$ be a smooth, proper and geometrically irreducible curve over a non-archimedean local field $k$. Then the invariant map \eqref{inv} induces an isomorphism of abelian groups
\[ 
\bxk\isoto I(\be X\lle)^{-1}\Z/\le\Z\,.
\]
Equivalently, $\bxk$ is a group of order $I(\be X\lle)$.
\end{theorem}
\begin{proof} If $k$ is a $p\e$-adic field (i.e., a finite extension of $\Q_{\e p}$), then the above statement is equivalent to \cite[Theorem 3, p.~130]{lich} (see Remark \ref{bel} below). If $k$ is a local function field (i.e., a finite extension of $\mathbb{F}_{\be p}((t))$), then the proof is formally the same as that of \cite[Theorem 3, p.~130]{lich} using Milne's extension to the local function field case of Tate's duality theorem for abelian varieties over $p\,$-\e adic fields \cite[Theorem III.7.8, p.~285]{adt}. Here the relevant abelian variety is  $\picor\!$, which is indeed an abelian variety by the smoothness of $X$ \cite[Corollary 3.2]{fga236}.	
\end{proof}

\begin{remark} \label{bel} Regarding the statement of Theorem \ref{roq}\,, in \cite[Theorem 3, p.~130]{lich} the term {\it connected} and the group $\br(k(\lbe X)/\le k)$, where $k(\lbe X)$ is the function field of $X$, appear in place of the term {\it irreducible} and the group $\bxk$, respectively. Since $X$ is smooth, connectedness and irreducibility are equivalent concepts for $X$ \cite[Remark 6.37, p.~165]{gw}. Further, the pullback map $\br\le X\to \br\e k(\lbe X)$ is injective by \cite[Corollary 1.8, p.~170]{gb}. Thus Theorem \ref{roq} above (over a $p\e$-adic field) is indeed equivalent to \cite[Theorem 3, p.~130]{lich}.
\end{remark}

\smallskip

\begin{theorem} \label{rgn} Let $k$ be a non-archimedean local field and let $X$ be a proper and geometrically integral $k\e$-\le curve such that $X^{\lle\be N}$ is {\rm smooth} over $k$. Then
$\bxk$ is a group of order $I(\be X\lle)$.
\end{theorem}
\begin{proof} Recall the normalization morphism $\nu\colon X^{\lle\be N}\to X$ and let $Y$ be the non-normal locus of $X\le$. Since $X^{\lle\be N}$ is a proper and geometrically integral FA $k$-scheme (see the proof of Theorem \ref{sn}), the associated conductor square
\[
\xymatrix{X^{\lle\be N}\be\times_{\lbe X}\! Y \,\ar@{^{(}->}[r]\ar[d]_{\nu\le \times_{\lbe X}\lbe Y}&
X^{\lle\be N}\ar[d]_(.45){\nu}\\
Y\,\ar@{^{(}->}[r]&X
}
\]
is a pinching diagram of type \eqref{pin}. See \cite[Lemma 3.1 and its preamble]{lau} for more details. Now Theorems \ref{gbk2} and \ref{roq} and the isomorphism \eqref{inv2} show that the invariant map \eqref{inv} induces an isomorphism of abelian groups
\[
\bxk\isoto \left(\!I\!\lbe\left(\lbe X^{\lle\be N}\right)^{\!-1}\!\Z\lle/\le\Z\right)\cap \bigcap_{\e y\e\in\e Y}\!\left([\kappa(y)\colon\be k\e]^{-1}\Z/\le\Z\right)\be,
\]
where the intersections take place in $\Q/\Z$. Thus $\#\bxk=\gcd\{I\!\lbe\left(\be X^{\be N}\right)\be,I\lbe(Y\lle)\}$, which is divisible by $I(\be X\lle)$ \eqref{divv}. Since $\#\bxk$ divides $I(\be X\lle)$, we conclude that $\#\bxk=\gcd\{I\!\lbe\left(\be X^{\be N}\right)\be,I\lbe(Y\lle)\}=I(\be X\lle)$, as claimed.
\end{proof}

\begin{remark}\label{com} Let $k$ be any field of characteristic exponent $p$ and let $X$ be an algebraic $k$-scheme. Then $X^{\lle\be N}$ is geometrically normal if, and only if, $(\lbe X^{\lle\be N}\le)_{k^{1\lbe/\lbe p}}$ is normal (see \cite[${\rm IV}_{\be 4}$, Proposition 6.7.7]{ega} and \cite[Proposition 2.10(3)]{tan}). Thus the hypothesis of Theorem \ref{rgn} certainly holds if $k$ is a $p\e$-adic field but not, in general, if $k$ is a local function field (see Example \ref{ex3}). Thus the methods of this paper are insufficient for obtaining a precise formula for $\#\bxk$ (if such a formula exists) if $X$ is an arbitrary proper and geometrically integral curve over a local function field $k$.
\end{remark}


\begin{thebibliography}{[48]}
	
\bibitem[Ach]{ach} Achet, R.:\emph{	The Picard group of the forms of the affine line and of the additive group.} J. Pure Appl. Algebra \textbf{221} (2017), no. 11, 2838--2860. 
	
\bibitem[ABBB]{a3b} Auel, A., Bigazzi, A., B\"ohning, C. and von Bothmer, H-C.:\emph{ Universal triviality of the Chow group of $0$-cycles and the Brauer group}. Int. Math. Res. Not. IMRN \textbf{2021}, no. 4, 2479--2496.

\bibitem[ACTP]{actp}  Auel, A., Colliot-Th\'el\`ene, J.-L. and Parimala, R.:\emph{ Universal unramified cohomology of cubic fourfolds containing a plane.} Brauer groups and obstruction problems, 29--55, Progr. Math., \textbf{320}, Birkh\"auser/Springer, Cham, 2017.
	
	
\bibitem[BK]{bk} Ballico, E. and Koll\'ar, J.\emph{ The Picard group of singular curves}.
Abh. Math. Sem. Univ. Hamburg {\textbf{73}} (2003), 225--227. 



\bibitem[Bey]{bey} Beyl, R.: \emph{ The connecting morphism in the kernel-cokernel sequence.} Arch. der Math. {\textbf{32}} (1979), no. 4, 305--308.


\bibitem[Bo]{bo} Borel, A.: \emph{ Linear Algebraic Groups.} Second edition. Graduate Texts in Mathematics, {\textbf{126}}. Springer-Verlag, New York, 1991.

 
\bibitem[BLR]{blr} Bosch, S., L\"utkebohmert, W. and Raynaud, M.:
N\'eron models. Ergebnisse der Mathematik und ihrer Grenzgebiete,
21. Springer-Verlag, Berlin, 1990.

\bibitem[Bou]{bca} Bourbaki, N.:\emph{ Algebra II.}
Chapters 4-7. Translated from the 1981 French edition by P. M. Cohn and J. Howie. Reprint of the 1990 English edition. Elements of Mathematics (Berlin). Springer-Verlag, Berlin, 2003.



\bibitem[BL20]{bl20} Bright, M. and Loughran, D.\emph{ Brauer-Manin obstruction for Erd\H{o}s-Straus surfaces}. Bull. Lond. Math. Soc. \textbf{52} (2020), no. 4, 746--761. 




\bibitem[Bri15]{bri15} Brion, M.:\emph{ Which algebraic groups are Picard varieties?} Sci. China Math. \textbf{58} (2015), no. 3, 461--478.

\bibitem[Ces19]{ces19}  \v{C}esnavi\v{c}ius, K.:\emph{ Purity for the Brauer group.} Duke Math. J. \textbf{168} (2019), no. 8, 1461--1486. 

 
\bibitem[CTS87]{cts87} Colliot-Th\'el\`ene, J.-L. and Sansuc, J.-J.: \emph{
La descente sur les vari\'et\'es rationnelles, II.} Duke Math. J. \textbf{54} (1987), no. 2, 375--492. 

\bibitem[CTSk]{ctsk} Colliot-Th\'el\`ene, J.-L. and Skorobogatov, A.:\emph{ The Brauer-Grothendieck group}.
Ergebnisse der Mathematik und ihrer Grenzgebiete, 3. Folge, \textbf{71}. Springer, Cham, 2021.

\bibitem[CGP]{cgp} Conrad, B., Gabber, O. and Prasad, G.: \emph{Pseudo-reductive groups}. Second Ed. New Math. Monograps \textbf{26}, Cambridge U. Press, 2015.

\bibitem[$\text{SGA3}_{\le\text{new}}$]{sga3}  Demazure, M. and
Grothendieck, A. (Eds.): Sch\'emas en groupes. S\'eminaire de
G\'eom\'etrie Alg\'ebrique du Bois Marie 1962-64 (SGA 3). Augmented and
corrected 2008-2011 re-edition of the original by P.Gille and P.Polo.
Available at \url{http://www.math.jussieu.fr/~polo/SGA3}. Reviewed
at \url{http://www.jmilne.org/math/xnotes/SGA3r.pdf}.


\bibitem[Eis]{eis} Eisenbud, D.:\emph{ Commutative algebra. With a view toward algebraic geometry.} Graduate Texts in Mathematics, \textbf{150}. Springer-Verlag, New York, 1995.


\bibitem[Fer]{fer}  Ferrand, D.:\emph{ Conducteur, descente et pincement}. Bull. Soc. Math. France \textbf{131} (2003), no. 4, 553--585. 


\bibitem[SA]{sa} Ford, T.J.:\emph{ Separable algebras.}
Graduate Studies in Mathematics, {\textbf{183}}. American Mathematical Society, Providence, RI, 2017.


\bibitem[GLL]{gll} Gabber, O., Liu, Q. and Lorenzini, D.:\emph{ The index of an algebraic variety.} Invent. Math. \textbf{192} (2013), no. 3, 567--626.

\bibitem[Gei09]{gei} Geisser, T.:\emph{ The affine part of the Picard scheme.} Compos. Math. \textbf{145} (2009), no. 2, 415--422, and \emph{Corrigendum}. Ibid. \textbf{157} (2021), no. 10, 2338-2340. See also \url{https://www2.rikkyo.ac.jp/web/geisser/publications.html}.


\bibitem[GS]{gs} Gille, P. and Szamuely, T.:\emph{ Central Simple Algebras and Galois Cohomology.} Second edition. Cambridge Studies in Advanced Mathematics \textbf{165}, Cambridge University Press, 2017.


\bibitem[GA17]{ga17}  Gonz\'alez-Avil\'es, C.D.: \emph{ On the group of units and the Picard group of a product.} Eur. J. Math. {\textbf{3}} (2017), no. 3, 471--506.


\bibitem[GA18]{ga18}  Gonz\'alez-Avil\'es, C.D.: \emph{ The units-Picard complex and the Brauer group of a product.} J. Pure Appl. Algebra {\textbf{222}} (2018), no. 9, 2746--2772.



\bibitem[GW]{gw} G\"ortz, U. and Wedhorn, T.:\emph{ Algebraic geometry I. Schemes--with examples and exercises}. Second edition. Springer Studium Mathematik--Master. Springer Spektrum, Wiesbaden, 2020.



\bibitem[GB]{gb} Grothendieck, A.:\emph{ Le groupe de Brauer}. In: Dix Expos\'es sur la cohomologie des sch\'emas, 46-188, Adv. Stud. Pure Math., {\bf{3}}, North-Holland, Amsterdam, 1968. 


\bibitem[$\text{EGA I}_{\le\text{new}}$\e]{ega1} Grothendieck, A. and
Dieudonn\'e, J.: \'El\'ements de g\'eom\'etrie alg\'ebrique I. Grundlehren der mathematischen Wissenschaften {\textbf{166}}. Springer-Verlag, Berlin, 1971.

\bibitem[EGA]{ega} Grothendieck, A. and Dieudonn\'e, J.: \'El\'ements de g\'eom\'etrie alg\'ebrique.  Grund. der Math. Wiss. {\bf{166}} (1971) ($=\text{EGA I}_{\le\text{new}}$), Publ. Math. IHES {\bf{8}} ($=\text{EGA II}$) (1961), {\bf{11}} $\text{III}_{1}$ (1961), {\bf{20}} ($=\text{EGA IV}_{1}$) (1964), {\bf{24}} ($=\text{EGA IV}_{2}$) (1965), {\bf{32}} ($=\text{EGA IV}_{4}$) (1967).


\bibitem[G]{fga236} Grothendieck, A.\emph{ Technique de descente et th\'eor\`emes d’existence en g\'eom\'etrie
alg\'ebrique. {\rm{VI}}. Les sch\'emas de Picard : propri\'et\'es g\'en\'erales.}
S\'eminaire N. Bourbaki, 1961-1962, expos\'e no. \textbf{236}, p. 221-243.


\bibitem[IIL20]{iil20} Ito, K., Ito, T. and Liedtke, C.\emph{ Deformations of rational curves in positive characteristic.} J. Reine Angew. Math. \textbf{769} (2020), 55--86. 


\bibitem[KMT]{kmt} Kambayashi, T., Miyanishi, M. and Takeuchi, M.:\emph{ Unipotent algebraic groups.} Lecture Notes in Mathematics \textbf{414}, Springer--Verlag, Berlin--New York, 1974. 


\bibitem[Klei]{klei} Kleiman, S.:\emph{ The Picard scheme}. Fundamental algebraic geometry, 235--321, Math. Surveys Monogr., \textbf{123}. Amer. Math. Soc., Providence, RI, 2005.



\bibitem[KK]{kk} Koll\'ar, J.\emph{ Singularities of the minimal model program. With a collaboration of S\'andor Kov\'acs.} Cambridge Tracts in Mathematics \textbf{200}. Cambridge University Press, Cambridge, 2013.


\bibitem[Lau]{lau} Laurent, B.:\emph{ Almost homogeneous curves over an arbitrary field}. 
Transform. Groups \textbf{24} (2019), no. 3, 845--886. 


\bibitem[Lich]{lich} Lichtenbaum, S. \emph{ Duality theorems for curves over $p$-adic fields}. Invent. Math. \textbf{7} (1969), 120--136.



\bibitem[MiEt]{miet} Milne, J.S.: \'Etale Cohomology.
Princeton University Press, Princeton, 1980.

\bibitem[ADT]{adt} Milne, J.S.: Arithmetic Duality
Theorems. Second Edition (electronic version), 2006.



\bibitem[Ov21]{ov21} Overkamp, O.:\emph{ On Jacobians of geometrically reduced curves and their N\'eron models.} \url{https://arxiv.org/abs/2110.00545}


\bibitem[Roq]{roq} Roquette, P.:\emph{ Splitting of algebras by function fields of one variable.} Nagoya Math. J. \textbf{27} (1966), 625--642. 



\bibitem[Rus70]{rus} Russell, P.:\emph{ Forms of the affine line and its additive group.} Pacific J. Math. \textbf{32} (1970), 527--539.


\bibitem[SS]{ss} Serbinowsky, B. and Schwede, K. \emph{ Seminormalization package for Macaulay}2. \url{https://arxiv.org/abs/1809.09256}


\bibitem[CG]{scg} Serre, J.-P.:\emph{ Cohomologie Galoisienne.} Lecture Notes in Math. \textbf{5}, Springer-Verlag, New York, 1997.


\bibitem[Ser2]{ser2} Serre, J.-P.:\emph{ Local Fields}. Translated from the French by Marvin Jay Greenberg. Graduate Texts in Mathematics, \textbf{67}, Springer-Verlag, New York-Berlin, 1979 (second corrected printing, 1995).


\bibitem[Sh]{sh} Shatz, S. \emph{ Profinite groups, arithmetic, and
Geometry.} Annals of Mathematics Studies {\textbf{67}}. Princeton University Press, Princeton, N.J.; University of Tokyo Press, Tokyo, 1972.


\bibitem[T]{t} Tamme, G.: Introduction to \'Etale Cohomology. Translated from the German by Manfred Kolster. Universitext. Springer-Verlag, Berlin, 1994.

\bibitem[Tan]{tan} Tanaka, H.:\emph{ Invariants of algebraic varieties over imperfect fields.} Tohoku Math. J. (2) \textbf{73} (2021), no. 4, 471--538.
 
\bibitem[SP]{sp} The Stacks Project. \url{http://stacks.math.columbia.edu}

\bibitem[To]{to} Totaro, B.: \emph{Pseudo-abelian varieties}. Ann. Sci. \'Ec. Norm. Sup\'er. (4) \textbf{46} (2013), no. 5, 693--721. 

\end{thebibliography}
\end{document}